\documentclass[12pt]{amsart}
\usepackage{amssymb}
\usepackage[german,english]{babel}

\theoremstyle{plain}
\newtheorem{theorem}{Theorem}[section]
\newtheorem{proposition}[theorem]{Proposition}
\newtheorem{corollary}[theorem]{Corollary}
\newtheorem{lemma}[theorem]{Lemma}

\theoremstyle{definition}
\newtheorem{definition}[theorem]{Definition}

\newtheorem{remark}[theorem]{Remark}

\newtheorem{conjecture}[theorem]{Conjecture}

\theoremstyle{remark}

\numberwithin{equation}{section}

\newcommand{\N}{\mathbb N}
\newcommand{\Z}{\mathbb Z}
\newcommand{\Q}{\mathbb Q}
\newcommand{\R}{\mathbb R}
\newcommand{\C}{\mathbb C}

\newcommand{\Id}{\mathrm{Id}}
\newcommand{\GL}{\operatorname{GL}}
\newcommand{\Ot}{\operatorname{O}}
\newcommand{\SO}{\operatorname{SO}}
\newcommand{\Spin}{\operatorname{Spin}}

\newcommand{\Hom}{\operatorname{Hom}}

\newcommand{\so}{\mathfrak{so}}
\newcommand{\spin}{\mathfrak{spin}}

\newcommand{\diag}{\operatorname{diag}}
\newcommand{\ba}{\backslash}

\newcommand{\norma}[1]{\|{#1}\|_1}
\newcommand{\ZZ}{\mathbb{E}}
\newcommand{\menos}{R}
\newcommand{\sgn}[1]{\operatorname{sgn}(#1)}
\newcommand{\mult}{\operatorname{mult}}
\newcommand{\rr}{r}
\newcommand{\gammab}{\overline{\gamma}}

\title[Dirac multiplicities on lens spaces]{An explicit formula for the Dirac multiplicities on lens spaces}
\author{Sebastian Boldt, Emilio A. Lauret}
\address{Institut f\"ur Mathematik, Humboldt-Universit\"at zu Berlin, D-10099 Berlin, Germany}
\email{boldt@math.hu-berlin.de}
\address{FaMAF--CIEM \\ Universidad Nacional de C\'ordoba\\ 5000-C\'ordoba, Argentina.}
\email{elauret@famaf.unc.edu.ar}
\date{December 8, 2014}
\thanks{Boldt was supported by DFG Sonderforschungsbereich 647 and Lauret by CONICET, FONCyT, Secyt-UNC and the Oberwolfach Leibniz Fellows Program}

\begin{document}

\maketitle

\begin{abstract}
We present a new description of the spectrum of the (spin-) Dirac operator $D$ on lens spaces. Viewing a spin lens space $L$ as a locally symmetric space $\Gamma\ba \Spin(2m)/\Spin(2m-1)$ and exploiting the representation theory of the $\Spin$ groups, we obtain explicit formulas for the multiplicities of the eigenvalues of $D$ in terms of finitely many integer operations. As a consequence, we present conditions for  lens spaces to be Dirac isospectral. Tackling classic questions of spectral geometry, we prove with the tools developed that neither spin structures nor isometry classes of lens spaces are spectrally determined by giving infinite families of Dirac isospectral lens spaces.
These results are complemented by examples found with the help of a computer.
\end{abstract}

\tableofcontents

\section{Introduction}

Let $M$ be a compact spin Riemannian manifold and fix spin structure on it.
The then canonically given (spin-) Dirac operator is an elliptic, essentially self adjoint first order differential operator.
Its spectrum is therefore real, consists only of eigenvalues with finite multiplicities and has $\pm\infty$ as the only accumulation points.

The determination of Dirac spectra of compact spin Riemannian manifolds is usually only possible in the presence of enough symmetries. As such there are only few explicit computations of Dirac spectra (cf.\ \cite[\S1.1.2]{Bar06}, \cite[\S2]{Ginoux}).
The aim of this paper is to contribute to this problem in the class of lens spaces, which are spherical space forms with cyclic fundamental group.

On the sphere, Hitchin~\cite{Hitchin} was the first to determine the Dirac spectrum in dimension $3$, S.~Sulanke computed it in her Ph.D.\ thesis \cite{Sulanke} in all dimensions, C.~B\"ar gave an alternative proof in his Ph.D.\ thesis \cite{Bar91} (see also \cite{Bar96}).
The Dirac spectrum on $S^n$ consists of the eigenvalues $\pm \left(k + \frac n2 \right)$ with corresponding multiplicities $2^{\lfloor\frac n2\rfloor}\binom {k+n-1}{n-1}$, $k\ge 0$.

We now consider a spherical space form, that is, $M_\Gamma:=\Gamma\ba S^n$ with $\Gamma\subset\SO(n+1)$ acting freely on $S^n$.
We restrict our attention to the case $n=2m-1$ since the only manifold properly covered by an even dimensional sphere is the real projective space.
Spin structures on $M_\Gamma$ are in 1--1 correspondence with homomorphisms $\tau:\Gamma\to\Spin(2m)$ satisfying $\Theta\circ\tau = \operatorname{Id}_\Gamma$, where $\Theta:\Spin(2m)\to\SO(2m)$ is the universal covering homomorphism.
The Dirac spectrum of $M_\Gamma$ endowed with $\tau$ is contained in the Dirac spectrum of the sphere, though its multiplicities are in general smaller than the corresponding ones of the sphere.
More precisely, the multiplicity $\mult_{(M_\Gamma,\tau)}(\pm\lambda_k)$ of the eigenvalue $\pm\lambda_k$ ($\lambda_k:=k+\tfrac n2=k+m-\tfrac12$) satisfies
$$
0\leq  \mult_{(M_\Gamma,\tau)}(\pm\lambda_k) \leq 2^{m-1}\tbinom{k+2m-2}{2m-2}.
$$

B\"ar~\cite{Bar96}, inspired by Ikeda's extensive work on the Laplace spectrum of spherical space forms, approached the spectrum of $M_\Gamma$ endowed with $\tau$ by giving closed formulas for the generating functions $F_{(M_\Gamma,\tau)}^\pm(z):=\sum_{k\ge 0} \mult_{(M_\Gamma,\tau)}(\pm\lambda_k)\, z^k$ which read
\begin{equation}\label{eq1:F^pm}
F_{(M_\Gamma,\tau)}^\pm(z) =
  \frac{1}{|\Gamma|} \sum_{\gamma\in\Gamma}
  \frac{\chi^\mp(\tau(\gamma)) - z\cdot\chi^\pm(\tau(\gamma))} {\det(1_{2m}-z\gamma)},
\end{equation}
where $\chi^\pm$ denote the half spin characters.
Formula \eqref{eq1:F^pm} can be used to compute the Dirac spectrum in particular cases (see for instance \cite[\S7]{MPT} and \cite[\S9--10]{Teh}).
In general, however, it does not suit the task.

Our goal is to give an alternative and explicit description of the Dirac spectrum on lens spaces.
The approach is representation theoretic in nature, essentially using Frobenius reciprocity.
In fact it is the same as the one used in \cite{LMR}, where R.~Miatello, J.P.~Rossetti and the second named author treated the Laplace operator case, obtaining an explicit description of the spectrum in terms of one-norm lengths of certain associated congruence lattices and also a geometric characterization of isospectrality.

To each $(2m-1)$-dimensional lens space $L:=L(q;s_1,\dots,s_m)$ (see Subsection \ref{subsec:lens} for notation) admitting a spin structure $\tau$, we associate an \emph{affine congruence lattice} $\mathcal L$ (see Lemma~\ref{lem4:dimV-lens}) inside the affine lattice $\ZZ^m:=(\tfrac12,\dots,\tfrac12)+\Z^m$.
For example, when the order $q$ of the fundamental group is odd, $L$ admits only one spin structure $\tau$ and its associated affine congruence lattice is
\begin{equation}\label{eq1:mathcalL}
\mathcal L=\left\{\tfrac12(a_1,\dots,a_m)\in \ZZ^m: a_1s_1+\dots+a_ms_m\equiv 0\pmod q\right\}.
\end{equation}
The main result, Theorem~\ref{thm4:dimV-lens}, states that
\begin{align}\label{eq1:multiplicity}
\mult_{(L,\tau)}(-\lambda_k) &= \sum_{r=0}^k \tbinom{r+m-2}{m-2} N_{\mathcal L}(r,k-r),\\
\mult_{(L,\tau)}(+\lambda_k) &= \sum_{r=0}^k \tbinom{r+m-2}{m-2} N_{\mathcal L}(r+1,k-r), \notag
\end{align}
where $N_{\mathcal L}(r,k)$ denotes the number of tuples $\mu=\tfrac12(a_1,\dots,a_m)\in\mathcal L$ whose one-norm equals $k+\tfrac{m}{2}$ (i.e.\ $\norma{\mu}:=\tfrac12\sum |a_j|=k+\tfrac{m}{2}$) and whose number of negative entries is congruent to $r$ modulo $2$.

As a consequence of this formula, we obtain a characterization of Dirac isospectrality in the class of lens spaces (Corollary~\ref{cor4:char-lens}).
Namely, two lens spaces with fixed spin structures have the same Dirac spectra if and only if their associated affine congruence lattices $\mathcal L$ and $\mathcal L'$ are \emph{oriented $\norma{\cdot}$-isospectral}, that is, $N_{\mathcal L}(\epsilon,k)=N_{\mathcal L'}(\epsilon,k)$ for every $k\geq0$ and every $\epsilon=0,1$.

This characterization allows us to find families of Dirac isospectral lens spaces.
We present in Theorem~\ref{thm5:increasing} an infinite sequence of growing families of lens spaces which are pairwise non-isometric and Dirac isospectral.
Theorem~\ref{thm5:q-odd} gives an infinite sequence of pairs of $7$-dimensional non-isometric lens spaces which are Dirac isospectral.
When $q$ and $m$ are even, lens spaces have two inequivalent spin structures and Theorem~\ref{thm5:q-even} gives an infinite sequence of $7$-dimensional lens spaces such that their two spin structures produce the same Dirac spectra.

We also prove a finiteness result (Proposition~\ref{prop6:finiteness}) that allows for the previous characterizations to be implemented in a computer program.
We associate to each lens space endowed with a spin structure finitely many numbers in such a way they determine its Dirac spectrum.
In particular, one has to check only finitely many equations to decide whether two lens spaces are Dirac isospectral.
As a consequence, Corollary~\ref{cor6:finiteness-eigenv} ensures that two spin lens spaces $(L,\tau)$ and $(L',\tau')$ are Dirac isospectral if and only if their multiplicities $\mult_{(L,\tau)}(\pm\lambda_k)$ and $\mult_{(L',\tau')}(\pm\lambda_k)$ coincide for all $k<mq$, where $q$ is the order of the fundamental groups and $2m-1$ is the dimension of $L$ and $L'$.

Note that, for a given dimension and a given order of the fundamental group, both obstructions for isospectrality, there are only finitely many isometry classes of (oriented) lens spaces (with a fixed spin structure).
Together with the finiteness result mentioned above, one can therefore systematically search for Dirac isospectral lens spaces with the help of a computer, as we have done.
The corresponding results (see Table~\ref{table:examples}) can be found in Section~\ref{sec:comp} together with a discussion of some phenomena that occur.
The most important one is that we have not found any pair of non-isometric Dirac isospectral lens spaces in dimensions $n\equiv 1\pmod4$.
The authors believe that this is always the case (see Conjecture~\ref{conj6:dim5}).
Furthermore, in dimension $7$, most of the Dirac isospectral pairs are also $p$-isospectral for all $p$, though there are exceptions, which Remark~\ref{rem6:dim7} points out.

There are examples of Dirac isospectral spin Riemannian manifolds in the literature, but not in the class of lens spaces.
Actually, to our best knowledge, the following list contains all of the non-trivial examples: \cite{Bar96}, \cite{AB}, \cite{MP06} (cf.\ \cite[\S6.1]{Ginoux} and \cite{Podesta06}).
B\"ar pointed out in \cite[Cor.~2]{Bar96} that the usual construction of Laplace isospectral manifolds by using almost conjugate subgroups (known as ``generalized Sunada method'') also works for Dirac operators on spherical space forms admitting spin structures.
This method cannot be applied for non-isometric lens spaces since cyclic almost conjugate subgroups are necessarily conjugate, thus the associated manifolds are isometric (cf.\ \cite[p.~79]{Bar96}).
Moreover, our examples are fundamentally different from the ones produced by this method since its resulting manifolds are always strongly isospectral (isospectral for every natural differential operator on natural bundles) while it is well known that strongly isospectral lens spaces are necessarily isometric (see for instance \cite[Prop.~7.2]{LMR}).

The paper is organized as follows.
In Section~\ref{sec:preliminaries} we recall two well known facts on spin geometry, namely, the Dirac spectrum of spherical space forms and the spin structures admitted by lens spaces.
We then introduce the notion of an affine congruence lattice associated to a lens space in Section~\ref{sec:congafflattices}, showing that there is an isometry between lens spaces that relates their spin structures if and only if there is an one-norm preserving isometry between the associated affine congruence lattices.
In Section~\ref{sec:spectra} we describe the Dirac spectrum of a spin lens space and give a characterization of Dirac isospectrality in terms of one-norm isospectrality of affine congruence lattices.
By using the previous characterization, we give in Section~\ref{sec:families} two infinite sequences of pairs of Dirac isospectral lens spaces.
We end in Section~\ref{sec:comp} with a computational approach to isospectral Dirac lens spaces and some remarks about the results obtained.

\section*{Acknowledgements}
A large part of the research for this paper was carried out during two stays at the Mathematisches Forschungsinstitut Oberwolfach.
Both authors wish to thank MFO for the funding, its hospitality and excellent working conditions.
The first named author wishes to express his gratitude for the hospitality of Universidad Nacional de C\'ordoba (Argentina) during a stay in June-July 2014.
The second named author wishes to thank the support of the Oberwolfach Leibniz Fellow programme in May-July 2013 and in August-November 2014.
Furthermore, the authors wish to thank Dorothee Sch\"uth and Roberto Miatello for fruitful discussions.

\section{Preliminaries}\label{sec:preliminaries}
In this section we review well known facts about the Dirac operator on spherical space forms.
We give a description of the Dirac spectrum of an arbitrary spin spherical space.
We use the same representation theoretic technique already used by Sulanke~\cite{Sulanke}.
Note that this approach is different from the one in \cite{Bar96} which uses Killing spinors.
We end the section introducing lens spaces and the classification of their spin structures.

The Dirac operator is different from the usual natural differential operators. It depends on a spin structure which need not exist, nor need it be unique. This dependence is non-trivial, i.e.\ in general the spectrum changes with a change of the spin structure, see for example \cite{spinstrsurvey}. If, however, a manifold carries an orientation preserving isometry that takes one spin structure to another, the spectra of the associated Dirac operators are the same. Therefore, if one is looking for a manifold which is Dirac isospectral to itself when equipped with two inequivalent spin structures, one has to check the non-existence of such an isometry for the isospectrality to be non-trivial.
We therefore present a ``spin version'' of the well-known isometry classification of lens spaces (see Proposition~\ref{prop2:lens-isom}) in Proposition~\ref{prop2:lens-isom-explicit}.

In addition to the dependence on a spin structure, the Dirac operator also depends on an orientation. A change of orientation results in the Dirac operator becoming its own negative and hence flipping its spectrum around zero. When in search of Dirac isospectral manifolds, one therefore has to compare spectra modulo orientations, i.e.\ modulo reflection. We thus also state conditions for an isometry between lens spaces to be orientation preserving (respectively reversing) in Proposition~\ref{prop2:lens-isom-explicit}.

\subsection{Dirac spectrum of spherical spaces forms}
We denote by $S^n$ the $n$-dimensional sphere of constant curvature one.
A spherical space form is a manifold of the form $\Gamma\ba S^n$ where $\Gamma$ is a finite group of isometries that acts freely.
For the convenience of the reader we  briefly recall the computation of the Dirac spectrum on spherical space forms (see Sulanke~\cite{Sulanke} and B\"ar~\cite{Bar96}).
We will restrict to the odd dimensional case $n=2m-1$ since the only manifold properly covered by $S^{2m}$ is $\mathbb P\R^{2m}$, which is not orientable and so in particular not spin.

The sphere $S^{2m-1}$ has a symmetric space structure as $G/K$ with $G=\Spin(2m)$ and $K=\{g\in G: g\cdot e_{2m}=e_{2m}\}\simeq\Spin(2m-1)$.
The $\SO(2m-1)$--bundle of oriented orthonormal frames of $S^{2m-1}$ is given by $\SO(2m)$ with projection onto the last column and multiplication by the structure group from the right on the remaining $2m-1$ columns. A spin structure is thus given by $\Theta:\Spin(2m)\to\SO(2m)$, the universal covering homomorphism, and since $S^{2m-1}$ is simply connected, it is also the unique spin structure of the sphere up to equivalence.

Let $\rho:K\to\GL(\Sigma_{2m-1})$ be the spinor representation.
Here $\Sigma_{2m-1}$ is a $2^{m-1}$-dimensional complex vector space.
The group $\Spin(2m)$ acts from the left on the Hilbert space $L^2(S^{2m-1},\Sigma_{2m-1})$ of $L^2$-sections of the spinor bundle over $S^{2m-1}$ by $(g\cdot f)(x)=f(g^{-1}x)$.
By Frobenius reciprocity, this representation decomposes into the direct sum
\begin{equation}\label{eq2:L^2-decompS^n}
L^2(S^{2m-1},\Sigma_{2m-1})
    = \sum_{\pi\in \widehat G} V_\pi \otimes\Hom_{K} (V_\pi,\Sigma_{2m-1}),
\end{equation}
where $\widehat G$ denotes the class of irreducible representations of $G=\Spin(2m)$, $g\in G$ acts on each $V_\pi \otimes\Hom_{K} (V_\pi,\Sigma_{2m-1})$ by $g\cdot (v\otimes A)= \pi(g)v\otimes A$, and $v\otimes A\in V_\pi \otimes\Hom_{K} (V_\pi,\Sigma_{2m-1})$ induces the section $gK\mapsto A\circ\pi(g^{-1})(v)$ (see for instance \cite[Thm.~2]{Bar92}).

We now turn to spherical space forms, quotients $\Gamma\ba S^{2m-1}$ of the sphere by a finite group of isometries acting freely.
Consider the induced action of $\Gamma$ on $\SO(S^{2m-1})\simeq\SO(2m)$. If this action lifts to the spin structure $\Spin(2m)$ of $S^{2m-1}$, then the quotient by this action is a spin structure of $\Gamma\ba S^{2m-1}$ (see \cite[Prop.~1.4.2]{Ginoux}). Furthermore, every spin structure is of this form by \cite[Ch.~2.2]{Friedrich}. Since the bundle of positively oriented orthonormal frames as well as the spin structure of $S^{2m-1}$ are groups themselves, group actions of $\Gamma$ on $\Spin(2m)$ are in 1-1 correspondence with group homomorphisms $\tau:\Gamma\to\Spin(2m)$ such that $\Theta\circ\tau=\Id_\Gamma$.

The $L^2$-spinor bundle on $\Gamma\ba S^{2m-1}$ corresponds to $\tau(\Gamma)$-invariant elements in $L^2(S^{2m-1},\Sigma_{2m-1})$, thus
\begin{equation}\label{eq2:L^2-decompS^n}
L^2(\Gamma\ba S^{2m-1},\Sigma_{2m-1})
    = \sum_{\pi\in \widehat G} V_\pi^{\tau(\Gamma)} \otimes\Hom_{K} (V_\pi,\Sigma_{2m-1}),
\end{equation}
where $V_\pi^{\tau(\Gamma)} = \{v\in V_\pi: (\pi\circ\tau)(\gamma)v=v\quad\forall \gamma\in\Gamma\}$.

Odd dimensional spherical space forms are always orientable since any discrete subgroup $\Gamma$ of the isometry group $\Ot(2m)$ of $S^{2m-1}$ acting without fixed points is already a subgroup of $\SO(2m)$.
On $\Gamma\ba S^{2m-1}$, we will always consider the inherited orientation from the previously chosen orientation on $S^{2m-1}$.

Let $\mathfrak g=\spin(2m)$ and $\mathfrak k\simeq\spin(2m-1)$, the Lie algebras of $G$ and $K$ respectively.
It is well known that $\mathfrak g=\sum_{1\leq i<j\leq 2m} \R e_ie_j$ and $\mathfrak k=\sum_{1\leq i<j\leq 2m-1} \R e_ie_j$ as subspaces of the Clifford algebra $\mathfrak{Cl}(\R^{2m})$, where $\{e_1,\dots,e_{2m}\}$ denotes the canonical basis of $\R^{2m}$.
We consider the Cartan decomposition $\mathfrak g=\mathfrak k\oplus\mathfrak p$ and the inner product $\langle\cdot,\cdot\rangle$ on $\mathfrak g$ given by a multiple of the Killing form in such a way that its restriction to $\mathfrak p$ induces the round metric on $S^{2m-1}$ with constant curvature one.
It turns out that $\{X_k:=\frac{1}{2}e_ke_{2m}:1\leq k\leq 2m-1\}$ is an orthonormal basis of $\mathfrak p$.

We fix the standard maximal torus in $G$,
\begin{equation}\label{eq2:max_torus}
T=\left\{g=
\textstyle\sum_{j=1}^{m} \big(\cos\theta_j+\sin \theta_j \, e_{2j-1}e_{2j} \big) \in G:
\theta_1,\dots,\theta_m\in\R
\right\}.
\end{equation}
Its Lie algebra is
\begin{equation}\label{eq2:cartan}
\mathfrak t = \left\{X=\textstyle\sum_{j=1}^m \theta_j e_{2j-1}e_{2j}\in\spin(2m): \theta_1,\dots,\theta_m\in\R \right\}.
\end{equation}
Note that $g=\exp(X)$ and
\begin{equation}
  \Theta(g)= \diag\left(
\left[\begin{smallmatrix}\cos(2\theta_1)&-\sin(2\theta_1) \\ \sin(2\theta_1)&\cos(2\theta_1)
\end{smallmatrix}\right]
,\dots,
\left[\begin{smallmatrix}\cos(2\theta_m)&-\sin(2\theta_m) \\ \sin(2\theta_m)&\cos(2\theta_m)
\end{smallmatrix}\right]
\right)\in\SO(2m)
\end{equation}
if $g\in T$ and $X\in \mathfrak t$ are as above.

The Cartan subalgebra $\mathfrak h:=\mathfrak t\otimes_\R \C$ of $\mathfrak g_\C:=\mathfrak g\otimes_\R \C \simeq \so(2m,\C)$ is given as in \eqref{eq2:cartan} with $\theta_1,\dots,\theta_m\in\C$ and, in this case, we let $\varepsilon_j\in\mathfrak h^*$ be given by $\varepsilon_j(X)=2i\theta_j$ for $1\leq j\leq m$.
The weight lattice of $G$ is
\begin{equation}\label{eq2:P(G)}
P(G)=\left\{\textstyle\sum_{j=1}^m a_j\varepsilon_j: a_j\in\Z\;\forall j\quad\text{or} \quad a_j\in\tfrac12+\Z\;\forall j\right\}.
\end{equation}
We fix the standard system of positive roots, thus a weight $\sum_{j=1}^m a_j\varepsilon_j\in P(G)$ is dominant if and only if $a_1\geq\dots\geq a_{m-1}\geq |a_m|$.
We consider in $K$ the maximal torus $T\cap K$, thus the associated Cartan subalgebra is included in $\mathfrak h$.
Under this convention, the weight lattice of $K$ is
\begin{equation}\label{eq2:P(K)}
P(K)=\left\{\textstyle\sum_{j=1}^{m-1} a_j\varepsilon_j: a_j\in\Z\;\forall j\quad\text{or} \quad a_j\in\tfrac12+\Z\;\forall j\right\}
\end{equation}
and $\sum_{j=1}^{m-1}a_j\varepsilon_j\in P(K)$ is dominant if and only if $a_1\geq\dots \geq a_{m-1}\geq0$.
By the highest weight theorem, since $G=\Spin(2m)$ and $K\simeq\Spin(2m-1)$ are simply connected compact Lie groups, the irreducible unitary representations of $G$ and $K$ are in correspondence with the dominant elements in $P(G)$ and $P(K)$ respectively.
For example, the spinor representation $\rho$ of $K$ has highest weight $\frac12(\varepsilon_1+\dots+\varepsilon_{m-1})$.
For $\Lambda\in P(G)$  dominant, we denote by $\pi_\Lambda$ the associated representations in $\widehat G$.

\begin{proposition}\label{prop2:spectrum}
Let $M_\Gamma:=\Gamma\ba S^{2m-1}$ be a spherical space form with spin structure induced by $\tau:\Gamma\to\Spin(2m)$.
For each non-negative integer $k$, let $\pi_{k}^\pm$ be the irreducible representations of $\Spin(2m)$ with highest weights
$$
\Lambda_k^\pm := \tfrac12\big((2k+1)\varepsilon_1+\varepsilon_2+\dots+\varepsilon_{m-1}\pm\varepsilon_m\big).
$$
Then, the eigenvalues of the Dirac operator $D$ on the spinor bundle of $M_\Gamma$ are $\pm \lambda_k$ with $\lambda_k:=k+\frac{2m-1}{2}$, with multiplicity
$$
\mult_{(M_\Gamma,\tau)}(\pm\lambda_k) :=\dim V_{\pi_k^\mp}^{\tau(\Gamma)}.
$$
\end{proposition}

\begin{proof}
Let $\Lambda=\sum_{j=1}^m a_j\varepsilon_j\in P(G)$ dominant.
By the classical branching law from $\Spin(2m)$ to $\Spin(2m-1)$ (see for instance \cite[Thm.~8.1.4]{GoodmanWallach}), $\Hom_{K} (V_{\pi_\Lambda},\Sigma_{2m-1})\neq0$ if and only if $a_j\in\frac12+\Z$ for all $j$ and
$$
a_1\geq \tfrac12\geq a_2\geq\dots\geq a_{m-1}\geq \tfrac12\geq |a_m|.
$$
Moreover $\Hom_{K} (V_\pi,\Sigma_{2m-1})$ has dimension one.
Hence, the sum in \eqref{eq2:L^2-decompS^n} can be restricted to $\{\pi_k^\pm :k\geq0\}$.

The Dirac operator $D$ commutes with the action of $G$, thus $D$ leaves invariant $V_\pi^{\tau(\Gamma)} \otimes\Hom_{K} (V_\pi,\Sigma_{2m-1})$ in \eqref{eq2:L^2-decompS^n} for each $\pi$.
Moreover, from \cite[Prop.~1]{Bar92}, $D$ restricted to $V_\pi^{\tau(\Gamma)} \otimes\Hom_{K} (V_\pi,\Sigma_{2m-1})$ is equal to $\Id\otimes D_\pi$ where
\begin{equation}\label{eq2:D_pi}
D_\pi(A)=-\sum_{j=1}^{2m-1} e_j\cdot A\circ d\pi(X_j).
\end{equation}
Since $\Hom_{K} (V_\pi,\Sigma_{2m-1})$ is one-dimensional for $\pi \in \{\pi_k^\pm :k\geq0\}$, $D_{\pi_k^\pm}$ acts by a scalar $\lambda_k^\pm$ whose multiplicity is $\dim \big(V_{\pi_k^\pm}^{\tau(\Gamma)} \otimes\Hom_{K} (V_{\pi_k^\pm},\Sigma_{2m-1})\big) = \dim V_{\pi_k^\pm}^{\tau(\Gamma)}$.
It remains to show that $\lambda_k^\pm = \mp (k+\frac{2m-1}{2})$.

Denote by $A\in\Hom_{K} \big(V_{\pi_k^\pm},\Sigma_{2m-1})\big)$ the projection onto $\Sigma_{2m-1}$ along the other isotypical summands of $V_{\pi_k^\pm}$.
By the above, $\Hom_{K} (V_{\pi_k^\pm},\Sigma_{2m-1})\big)=\C A$.
It is thus enough to show that
\begin{equation}\label{eq2:D_pi2}
\sum_{j=1}^{2m-1} e_j\cdot A\circ d\pi_k^\pm(X_j)(u) = \pm (k+\tfrac{2m-1}{2})A(u)
\end{equation}
for a vector $u$ that is not in the kernel of $A$.

The representation $\pi_k^{\pm}$ has highest weight $\Lambda_k^\pm=k\varepsilon_1+\Lambda_0^\pm$, thus it is contained in the tensor product $\pi_{k\varepsilon_1}\otimes\pi_0^\pm$.
Here, $\pi_{k\varepsilon_1}$ is the representation of traceless symmetric degree $k$ tensors in $e_1,\ldots,e_{2m}$ and $\pi_0^\pm$ are the positive and negative half-spinor representations of $\Spin(2m)$.

Define
\[
	u=(e_{2m-1}\mp i e_{2m})^k\otimes v \in V_{\pi_{k\varepsilon_1}}\otimes V_{\pi_0^\pm}
\]
where $v$ is a highest weight vector for $\pi_0^\pm$.
The element $u$ lies in the weight space with weight $\pm k\varepsilon_m+\Lambda_0^\pm$ and spans this weight space since it is conjugate under the Weyl group to the weight space with highest weight $\Lambda_k^\pm$.

The element $X_{2m-1}$ is in the Cartan subalgebra $\mathfrak{h}$ of $\spin(2m)$ and $A(u)$ is a highest weight vector of $\Sigma_{2m-1}$, so
$d\pi_k^\pm(X_{2m-1})(u)=(\pm k\varepsilon_m+\Lambda_0^\pm)(X_{2m-1})\, u =\pm(k+\tfrac{1}{2})i \,u.$
Hence,
$$
(e_{2m-1} \cdot A\circ  d\pi_k^\pm (X_{2m-1}))(u) = \pm i(k+\tfrac12) e_{2m-1} \cdot A(u) = \pm (k+\tfrac12) A(u)
$$
since multiplication by $e_{2m-1}$ is $-i$ by the usual choice of Clifford multiplication.

For the remaining terms ($1\leq j\leq 2m-2$) one can check that $e_j\cdot A(d\pi(X_j)(u))=\pm\frac{1}{2}A(u)$ by writing $X_j$ as a suitable linear combination of root vectors of $\so(2m-1)\simeq\spin(2m-1)$ and $\so(2m)\simeq\spin(2m)$.

Combining the terms we obtain
\begin{eqnarray*}
\sum_{j=1}^{2m-1} e_j\cdot A\cdot d\pi(X_j)(u)
    & = & (\pm(2m-2)\tfrac{1}{2}\pm (k+\tfrac{1}{2}))A(u) \\
    & = & \pm (k+m-\tfrac{1}{2})A(u),
\end{eqnarray*}
which establishes \eqref{eq2:D_pi2} and concludes the proof.
\end{proof}

\subsection{Spin structures on lens spaces}\label{subsec:lens}
The manifolds covered by $S^{2m-1}$ with cyclic fundamental group are called \emph{lens spaces}.
They can be described as follows.
We associate to $q\in\N$ and $s_1,\dots,s_m\in\Z$ such that $s_j$ is coprime to $q$ for all $1\leq j\leq m$, the lens space
\begin{equation}\label{eq2:L(q;s)}
L(q;s_1,\dots,s_m) := \langle\gamma\rangle \ba S^{2m-1}
\end{equation}
where
\begin{equation}\label{eq2:gamma}
\gamma=
\diag\left(
\left[\begin{smallmatrix}\cos(2\pi{s_1}/q)&-\sin(2\pi{s_1}/q) \\ \sin(2\pi{s_1}/q)&\cos(2\pi{s_1}/q)
\end{smallmatrix}\right]
,\dots,
\left[\begin{smallmatrix}\cos(2\pi{s_m}/q)&-\sin(2\pi{s_m}/q) \\ \sin(2\pi{s_m}/q)&\cos(2\pi{s_m}/q)
\end{smallmatrix}\right]
\right)\in\SO(2m).
\end{equation}
The element $\gamma$ generates a cyclic group of order $q$ in $\SO(2m)$ that acts freely on $S^{2m-1}$.
The next result is well known (see for instance \cite[Ch.~V]{Co}).

\begin{proposition}\label{prop2:lens-isom}
Let $L=L(q;s_1,\dots,s_m)$ and $L'=L(q;s_1',\dots,s_m')$ be lens spaces.
Then the following assertions are equivalent.
\begin{enumerate}
  \item $L$ is isometric to $L'$.
  \item $L$ is diffeomorphic to $L'$.
  \item $L$ is homeomorphic to $L'$.
  \item There exist $\ell \in\Z$ coprime to $q$, $\epsilon_1,\dots,\epsilon_m \in \{\pm1\}$ and $\sigma$ a permutation of $\{1,\dots,m\}$ such that $\ell \epsilon_{j} s_j\equiv s_{\sigma(j)}'\pmod{q}$.
\end{enumerate}
\end{proposition}

We next describe the spin structures admitted by the lens space $L=L(q;s_1,\dots,s_m)=\Gamma\ba S^{2m-1}$.
If $q$ is odd, let $\tau:\Gamma\to\Spin(2m)$ be given by
\begin{equation}\label{eq2:spin-str-odd}
  \tau(\gamma^k) = \prod_{j=1}^m \left(\cos\left(\tfrac{(q+1)k s_j\pi}{q}\right) + \sin\left(\tfrac{(q+1)k s_j\pi}{q}\right) e_{2j-1}e_{2j}\right).
\end{equation}
One can check that $\Theta\circ\tau=\Id_\Gamma$, thus $\tau$ defines a spin structure on $L$.
We next assume $q$ and $m$ even.
Sometimes we will abbreviate $s=(s_1,\dots,s_m)$.
Let $h_{q;s}=\sum_j \lfloor\tfrac{s_j}{q}\rfloor$.
For $h\in\Z$, we let $\tau_h:\Gamma\to\Spin(2m)$ be given by
\begin{equation}\label{eq2:spin-str-even}
\tau_h(\gamma^k)
= (-1)^{k\left(h+h_{q;s}\right)} \; \prod_{j=1}^m \left(\cos(\tfrac{k s_j\pi}{q}) + \sin(\tfrac{k s_j\pi}{q})\, e_{2j-1}e_{2j}\right).
\end{equation}
Again we have that $\Theta\circ\tau_h=\Id_\Gamma$, thus $\tau_{h}$ defines a spin structure on $L$.
Clearly $\tau_h$ depends only on the parity of $h$.

\begin{remark}
We include the term $(-1)^{h_{q;s}}$ in \eqref{eq2:spin-str-even} since we want to associate the spin structure to a lens space independently of its parameters.
To be more precise, suppose that we define the spin structure as in \eqref{eq2:spin-str-even} without the term $(-1)^{h_{q;s}}$.
Then, the lens spaces $L=L(q;s_1,\dots,s_m)$ and $L'=L(q;s_1+q,s_2,\dots,s_m)$ coincide by definition, but the spin structures are switched, that is, $\tau_0=\tau_1'$ and $\tau_1=\tau_0'$.
\end{remark}

The following result due to Franc~\cite{Franc} classifies the spin structures on lens spaces.

\begin{theorem}\label{thm2:spin-structures}
Let $L=L(q;s_1,\dots,s_m)$.
If $q$ is odd, $L$ admits (up to equivalence) only one spin structure which is induced by \eqref{eq2:spin-str-odd}.
If $q$ is even, $L$ does not admit a spin structure when $m$ is odd, and admits two inequivalent spin structures when $m$ is even, which are induced by $\tau_0$ and $\tau_1$ as in \eqref{eq2:spin-str-even}.
\end{theorem}

From now on, unless otherwise stated, we assume that the spin structures on $L(q;s_1,\dots,s_m)$ are induced by $\tau$ as in \eqref{eq2:spin-str-odd} or \eqref{eq2:spin-str-even} depending on the parity of $q$.
Moreover, for simplicity, we shall simply indicate the spin structure by $\tau$, in place of writing ``the spin structure induced by $\tau$''.

For the convenience of the next sections, we will give a more explicit version of Proposition~\ref{prop2:lens-isom}, giving a particular isometry and what it does to the spin structures.
We will use the identification
$$
S^{2m-1} = \{(z_1,\dots,z_m)\in\C^m: |z_1|^2+\dots+|z_m|^2 =1\},
$$
thus an element $g=\diag\left(\left[\begin{smallmatrix}\cos(\theta_1)&-\sin(\theta_1) \\ \sin(\theta_1)&\cos(\theta_1)
\end{smallmatrix}\right],\dots, \left[\begin{smallmatrix}\cos(\theta_m)&-\sin(\theta_m) \\ \sin(\theta_m)&\cos(\theta_m)
\end{smallmatrix}\right]\right)$ in the maximal torus $T$ of $\SO(2m)$ acts by
$$
g\cdot (z_1,\dots,z_m) = (e^{i\theta_1}z_1,\dots,e^{i\theta_m}z_m).
$$
Hence, the lens space $L=L(q;s_1,\dots,s_m)$ is defined by $S^{2m-1}$ under the relation
$$
(z_1,\dots,z_m) \sim (e^{2\pi i s_1/q}z_1,\dots,e^{2\pi i {s_m}/{q}}z_m).
$$
We usually write $\pi:S^{2m-1}\to L$ for the projection.

For $\sigma$ a permutation of $\{1,\dots,m\}$ and $\epsilon_1,\dots,\epsilon_m\in\{\pm1\}$, we associate the function
\begin{equation}\label{eq2:Psi}
\begin{array}{rcl}
\Psi:S^{2m-1}&\longrightarrow& S^{2m-1} \\
(z_1,\dots,z_m)&\longmapsto&(w_1,\dots,w_m),
\end{array}
\qquad\text{where }\;
w_{\sigma(j)}=
\begin{cases}
  z_{j} &\text{if } \epsilon_{j}=1,\\
  \overline{z_{j}} &\text{if } \epsilon_{j}=-1.
\end{cases}
\end{equation}
Clearly $\Psi$ is an isometry of $S^{2m-1}$.
Furthermore, we can see that each conjugate coordinate in $w_{\sigma(j)}$ changes the orientation, thus $\Psi$ will preserve or reverse one fixed orientation on $S^{2m-1}$ depending on $\epsilon=1$ or $\epsilon=-1$ respectively, where $\epsilon:=\epsilon_1\dots\epsilon_m$.

\begin{proposition}\label{prop2:lens-isom-explicit}
Let $L=L(q;s_1,\dots,s_m)$ and $L'=L(q;s_1',\dots,s_m')$ be lens spaces and $\pi:S^{2m-1}\to L$ and $\pi':S^{2m-1}\to L'$ their projections.
Let $\sigma$ be a permutation of $\{1,\dots,m\}$, let $\epsilon_1,\dots,\epsilon_m\in\{\pm1\}$ and let $\Psi$ be the associated isometry of $S^{2m-1}$ as above.
Then, there is an isometry $\psi:L\to L'$ such that $\pi'\circ\Psi = \psi \circ \pi$ if and only if there exists $\ell\in\Z$ such that
\begin{equation}\label{eq2:iso_condition}
\ell \epsilon_j s_j\equiv s_{\sigma(j)}'\pmod q
\qquad\text{for all $j$}.
\end{equation}
In this case the isometry $\psi$ preserves or reverses the orientation according to $\epsilon=1$ or $\epsilon=-1$, where $\epsilon=\epsilon_1\dots\epsilon_m$.

Moreover, when this happens and $q$ is even, we have that $\psi:L\to L'$ takes $\tau_h$ to $\tau_{h'}'$, where $\rho_j:=({\ell\epsilon_{j}s_{j}-s_{\sigma(j)}'})/{q}$ and
$$
{h'}=h+h_{q;s}+h_{q;s'}+\textstyle\sum_j \rho_j.
$$
\end{proposition}

\begin{proof}
Set $\xi=e^{2\pi i/q}$.
We first suppose that such $\psi$ exists, thus $\pi'\circ\Psi (\xi^{s_1}z_1,\dots,\xi^{s_m}z_m)= \psi \circ \pi (\xi^{s_1}z_1,\dots,\xi^{s_m}z_m) = \psi \circ \pi (z_1,\dots,z_m) = \pi'\circ\Psi (z_1,\dots,z_m)$ for every $(z_1,\dots,z_m)\in S^{2m-1}$.
We write $\Psi(z_1,\dots,z_m)= (w_1,\dots,w_m)$, thus one can check that $\Psi(\xi^{s_1}z_1,\dots,\xi^{s_m}z_m) = (\xi^{\epsilon_{\sigma^{-1}(1)} s_{\sigma^{-1}(1)}}w_1,\dots,\xi^{\epsilon_{\sigma^{-1}(m)} s_{\sigma^{-1}(m)}}w_m)$.
Hence, there exists $\ell\in\Z$ such that
$$
(\xi^{\ell s_1'}w_1,\dots,\xi^{\ell s_m'}w_m) = (\xi^{\epsilon_{\sigma^{-1}(1)} s_{\sigma^{-1}(1)}}w_1,\dots,\xi^{\epsilon_{\sigma^{-1}(m)} s_{\sigma^{-1}(m)}}w_m),
$$
then $\epsilon_j s_j\equiv \ell s_{\sigma(j)}'\pmod q$ for all $j$.

We next prove the converse.
Condition (\ref{eq2:iso_condition}) above is equivalent to
\begin{eqnarray}\label{eq2:generator-simil}
\gamma^\ell=\Psi^{-1}\circ \gamma'\circ \Psi .
\end{eqnarray}
This implies that the function $\psi:L\to L'$ given by $ \psi(\left[p\right]_L)=\left[\Psi(p)\right]_{L'} $ is a well defined isometric bijection.

Each conjugate coordinate in \eqref{eq2:Psi} changes the orientation, thus $\psi$ will preserve or reverse the orientation according to the parity of the number of conjugate coordinates.

Finally, relation (\ref{eq2:generator-simil}) lifts to the spin level as
$$
\tau'_{\overline{h}}(\gamma^\ell) = \widetilde{\Psi}^{-1}\cdot \tau_{h}(\gamma')\cdot \widetilde{\Psi}
$$
where $\widetilde{\psi}\in\operatorname{Pin}(2m)$ is one of the lifts of $\psi\in\operatorname{O}(2m)$ which proves the claim on the behavior of the spin structures under $\psi$.
\end{proof}

\section{Affine Congruence Lattices}\label{sec:congafflattices}
In this section we will associate to each lens space with a fixed spin structure an affine congruence lattice.
We will prove a close relation between isometries of lens spaces and linear bijections of affine congruence lattices that preserve the one-norm.

We first fix some notation.
Write
\begin{equation}\label{eq3:E^m}
  \ZZ^m=(\tfrac12+\Z)^m=\{(a_1,\dots,a_m)\in\Q^m: 2a_j\in2\Z+1 \quad\forall\,1\leq j\le m\},
\end{equation}
the translation of $\Z^m$ by $(\tfrac12,\dots,\tfrac12)$.
We will usually write the elements in $\ZZ^m$ as $\tfrac12(a_1,\dots,a_m)$ with $a_1,\dots,a_m$ odd integer numbers.
For $q\in\N$ and $s=(s_1,\dots,s_m)\in\Z^m$ such that $s_j$ is coprime to $q$ for every $j$, we set
\begin{equation}\label{eq3:L(q;s)}
  \mathcal L(q;s)=\{\tfrac12 (a_1,\dots,a_m)\in\ZZ^m: \textstyle\sum_j a_js_j\equiv0\pmod q\}.
\end{equation}
Furthermore, if $q$ is even, we set
\begin{equation}\label{eq3:L(q;s;h)}
  \mathcal L(q;s;h)=\{\tfrac12(a_1,\dots,a_m)\in\ZZ^m: \textstyle\sum_j a_js_j\equiv hq\pmod{2q}\}
\end{equation}
for $h\in\Z$.
Note that $\mathcal L(q;s;h)$ depends only on the parity of $h$ and $\mathcal L(q;s)=\mathcal L(q;s;0)\cup \mathcal L(q;s;1)$ as a disjoint union.

Let us denote by $\norma{\mu}$ the \emph{one-norm} of $\mu=\tfrac12(a_1,\dots,a_m)$, that is, $\norma{\mu}=\tfrac12\sum_{j=1}^m |a_j|$.
We recall that $\varphi:\ZZ^m\to\ZZ^m$ is an isometry that preserves the one-norm (i.e.\ a $\norma{\cdot}$-isometry) if and only if it is a composition of a permutation and multiplication by $\pm1$ in each coordinate (see for instance \cite[Prop.~3.3]{LMR}).
More precisely,
\begin{equation}\label{eq3:one-norm-isometry}
\varphi(\tfrac12(a_1,\dots,a_m)) = \tfrac12(\epsilon_{\sigma^{-1}(1)} a_{\sigma^{-1}(1)} ,\dots, \epsilon_{\sigma^{-1}(m)} a_{\sigma^{-1}(m)}),
\end{equation}
where $\sigma$ is a permutation of $\{1,\dots,m\}$ and $\epsilon_1,\dots,\epsilon_m\in\{\pm1\}$.
We will say that $\varphi$ preserves orientation if $\epsilon:=\prod_{j=1}^m\epsilon_j=1$, or reverses orientation if $\epsilon=-1$.

Define $\menos(\mu)=\#\{1\leq j\leq m: a_j<0\}$ for $\mu=\tfrac12(a_1,\dots,a_m)\in\ZZ^m$.
For any subset $\mathcal L$ of $\ZZ^m$ and $\epsilon\in\Z$, we denote $\mathcal L_{\epsilon} = \{\mu\in\mathcal L:\menos(\mu)\equiv \epsilon\pmod2\}$, thus $\mathcal L=\mathcal L_0\cup\mathcal L_1$ as a disjoint union.
Note that
$$
\menos(\varphi(\mu))-\menos(\mu)\equiv \tfrac{\epsilon-1}2\pmod2,
$$
that is, $R(\varphi(\mu))$ and $R(\mu)$ have the same parity if and only if $\epsilon=1$.
Hence, a $\norma{\cdot}$-isometry $\varphi$ preserves orientation if and only if $\varphi(\ZZ^m_0)=\ZZ^m_0$ and $\varphi(\ZZ^m_1)=\ZZ^m_1$ .

The next two propositions will illuminate the relation between isometries of lens spaces with given spin structures and $\norma{\cdot}$-isometries of affine congruence lattices.
This relation is explained in Corollary~\ref{cor3:relation-isometries}.

\begin{proposition}\label{prop3:cong-set-isom-odd}
Let $q$ be a positive odd integer and let $s$ and $s'$ be in $\Z^m$ with all coordinates coprime to $q$.
Let $\sigma$ be a permutation of $\{1,\dots,m\}$, $\epsilon_1,\dots,\epsilon_m\in\{\pm1\}$ and let $\varphi$ be as in \eqref{eq3:one-norm-isometry}.
Then, $\varphi(\mathcal L(q;s))=\mathcal L(q;s')$ if and only if there exists $\ell\in\Z$ coprime to $q$ such that
$$
\ell \epsilon_j s_j\equiv s_{\sigma(j)}'\pmod q
\qquad \text{for every $j$}.
$$
\end{proposition}
\begin{proof}
Let us first prove the converse.
We assume that there exists such $\ell\in\Z$.
Let $\mu=\tfrac12(a_1,\dots,a_m)\in\mathcal L(q;s)$, thus $\sum_j a_js_j\equiv0\pmod q$.
We have that $\varphi(\mu)\in\mathcal L(q;s')$ since
\begin{align*}
\sum_j a_{\sigma^{-1}(j)} \epsilon_{\sigma^{-1}(j)}  s_j'
    &=\sum_j  a_{j} \epsilon_{j} s_{\sigma(j)}'
    \equiv \sum_j a_{j} \ell \epsilon_{j}^2  s_{j} \pmod q \\
    &\equiv \ell\sum_j  a_{j} s_{j}\equiv0\pmod q.
\end{align*}
This proves that $\varphi(\mathcal L(q;s))\subset \mathcal L(q;s')$.
The other inclusion is very similar.

We now suppose that $\varphi(\mathcal L(q;s))=\mathcal L(q;s')$.
We can assume that $s_j$ is odd for every $j$, since if not, we replace it by $s_j+q$.
Let $\mu_i=\tfrac12(s_i,q,\dots,q, -s_1, q, \dots,q)$ for $2\leq i\leq m$, where the coordinate $-s_1$ is in the $i$th entry.
One can easily check that $\mu_i\in\mathcal L(q;s)$.
Hence $\varphi(\mu_i)\in\mathcal L(q;s')$ and
\begin{align*}
 0\equiv s_i\epsilon_1 s_{\sigma(1)}'-s_1\epsilon_i s_{\sigma(i)}' + q\sum_{j\neq 1,i} \epsilon_js_{\sigma(j)}'\equiv s_i\epsilon_1 s_{\sigma(1)}'-s_1\epsilon_i s_{\sigma(i)}'\pmod q.
\end{align*}
We let $\ell\in\Z$ such that $s_1\ell\equiv \epsilon_1 s_{\sigma(1)}'\pmod q$, thus $\ell\epsilon_i s_i\equiv  s_{\sigma(i)}'\pmod q$ for all $1\le i \le m$.
\end{proof}

\begin{proposition}\label{prop3:cong-set-isom-even}
Let $q$ be a positive even integer and let $s$ and $s'$ be in $\Z^m$ with all coordinates coprime to $q$.
Let $\sigma$ be a permutation of $\{ 1,\ldots,m \}$, $\epsilon_1,\ldots,\epsilon_m \in \{ \pm 1 \} $, $\rho \in \{ 0, 1 \}$ and let $\varphi$ be as in \eqref{eq3:one-norm-isometry}.
Then $\varphi(\mathcal L(q;s;h))=\mathcal L(q;s';h+\rho)$ for all $h \in \{ 0 , 1 \}$ if and only if there exists $\ell \in \Z$ coprime to $q$ such that
$$
\ell \epsilon_j s_j \equiv s'_{\sigma(j)} \pmod q
\quad\text{for all $j$ and}\quad
\rho \equiv \sum_{j=1}^m\frac{\ell \epsilon_j s_j - s'_{\sigma(j)}}{q}\pmod 2.
$$
\end{proposition}

\begin{proof}
Suppose that such an $\ell$ exists.
Note that $\ell$ is odd since it is coprime to $q$.
Let $\rho_j=({\ell \epsilon_j s_j - s'_{\sigma(j)}})/{q}$, thus $ s_{\sigma(j)}'=\ell \epsilon_js_j -\rho_jq$.
Let $\mu = \tfrac12(a_1,\ldots,a_m)\in \mathcal L(q;s;h)$, thus $\sum_j a_js_j\equiv hq\pmod{2q}$.
Then
\begin{align*}
\sum_j a_{j} \epsilon_{j} s_{\sigma(j)}'
    &\equiv \sum_j a_{j} \epsilon_{j} (\ell \epsilon_js_j-\rho_jq) \pmod {2q}\\
    &\equiv \ell\sum_j a_{j} s_{j}-q\sum_j a_{j} \epsilon_{j} \rho_j \pmod {2q} \\
    &\equiv (h+\rho)q \pmod {2q}.
\end{align*}
In the last step we used that $\ell$, $a_j$ and $\epsilon_j$ are odd integers.
We have proved that $\varphi(\mu) \in \mathcal L(q;s';h+\rho)$, while the reversed inclusion is very similar.

We now assume that $\varphi(\mathcal L(q;s;h))=\mathcal L(q;s';h+\rho)$ for $h=0,1$.
Define
\begin{eqnarray*}
\eta &=& \tfrac 12(-s_2,s_1,-s_4,s_3,\ldots,-s_{m},s_{m-1}), \\
\mu_i &=& \eta+(-s_i,0,\ldots,0,s_1,0,\ldots,0) \quad\text{for all } 1<i\le m.
\end{eqnarray*}
Clearly, $\eta \in\mathcal L(q;s;0)$ and $\mu_i \in\mathcal L(q;s;0)$, then
$$
s_1 \epsilon_i s'_{\sigma(i)} - s_i \epsilon_1 s'_{\sigma(1)}\equiv 0 \pmod  q.
$$
Let $\ell\in\Z$ be such that $\ell s_1\equiv \epsilon_1 s'_{\sigma(1)}\pmod q$.
By the above, $\ell s_i\equiv \epsilon_i s'_{\sigma(i)}\pmod q$ for all $1\le i\le m$.
It remains to show that this $\ell$ fulfills the condition on $\rho$.
For this, let $\rho_i=\frac{\ell s_i-\epsilon_i s'_{\sigma(i)}}{q}$.
Since $\varphi(\eta)\in\mathcal L(q;s';\rho)$, we have
$$
-s_2\epsilon_1 s'_{\sigma(1)} + s_1 \epsilon_2 s'_{\sigma(2)} + \ldots - s_m \epsilon_{m-1} s'_{\sigma(m-1)}
 + s_{m-1} \epsilon_{m} s'_{\sigma(m)} \equiv \rho q \pmod q.
$$
Multiplying both sides by $\ell$ and substituting $\ell s_i$ by $q \rho_i + \epsilon_i s'_{\sigma(i)}$ gives $\rho \equiv \sum_i \rho_i \pmod {2q}$.
\end{proof}

In the next result, $\epsilon$-oriented means that it preserves orientation if $\epsilon=1$ and it reverses orientation if $\epsilon=-1$.

\begin{corollary}\label{cor3:relation-isometries}
Let $q$ be a positive integer and let $s$ and $s'$ be in $\Z^m$ with all coordinates coprime to $q$.
Then, if $q$ is odd, there is an $\epsilon$-oriented isometry $\psi$ between  $L(q;s)$ and $L(q;s')$  if and only if there is an $\epsilon$-oriented $\norma{\cdot}$-isometry $\varphi$ between $\mathcal L(q;s)$ and $\mathcal L(q;s')$.
If $q$ is even and $\rho\in\{0,1\}$, there is an $\epsilon$-oriented isometry $\psi$ between  $L(q;s)$ and $L(q;s')$ taking the spin structures $\tau_h$ to $\tau'_{h+\rho}$ for $h=0,1$ if and only if there is an $\epsilon$-oriented $\norma{\cdot}$-isometry $\varphi$ with $\varphi (\mathcal L(q;s;h))=\mathcal L(q;s;h+\rho)$ for $h=0,1$.
\end{corollary}

\begin{proof}
Let $q$ be an odd integer. The claim follows from Proposition~\ref{prop2:lens-isom-explicit} and Proposition~\ref{prop3:cong-set-isom-odd} since both conditions are equivalent to the existence of a permutation $\sigma$ of $\{1,\dots,m\}$, $\epsilon_1,\dots,\epsilon_m\in\{\pm1\}$ and $\ell\in\Z$ such that $\ell\epsilon_j s_j\equiv s_{\sigma(j)}'\pmod q$ for all $j$.
If $q$ is even, the claim follows similarly from Proposition~\ref{prop2:lens-isom-explicit} and Proposition~\ref{prop3:cong-set-isom-even}.
\end{proof}

\section{Dirac spectra of lens spaces}\label{sec:spectra}

In this section we will give a geometric characterization of lens spaces that are Dirac isospectral in terms of the $\norma{\cdot}$-lengths of their associated affine congruence lattices.

We will identify an element $\sum_{j=1}^m a_j\varepsilon_j\in\mathfrak h^*$ (see \eqref{eq2:cartan}) with $(a_1,\dots,a_m)\in\C^m$.
Hence, the weight lattice $P(G)$ for $G=\Spin(2m)$ is the disjoint union of $\Z^m$ and $\ZZ^m=(\tfrac12,\dots,\tfrac12)+\Z^m$ (see \eqref{eq3:E^m}).
We recall that $\pi_k^\pm$ denotes the irreducible representation of $G$ with highest weight $\Lambda_k^\pm=\frac12((2k+1)\varepsilon_1+\varepsilon_2+\dots+\varepsilon_{m-1}\pm\varepsilon_m)$.
As usual, $m_\pi(\mu):=\dim V_\pi(\mu)$, the multiplicity of $\mu$ in $\pi$.

\begin{lemma}\label{lem4:dimV-lens}
Let $L(q;s_1,\dots,s_m)=\Gamma\ba S^{2m-1}$ be lens space with spin structure $\tau$.
Then
\begin{equation*}
\dim V_{\pi_k^\pm}^{\tau(\Gamma)}=
\sum_{\mu\in \mathcal L_{\Gamma;\tau}}\, m_{\pi_k^\pm}(\mu)
\end{equation*}
where
$$
\mathcal L_{\Gamma;\tau} =
\begin{cases}
  \mathcal L(q;s) &\text{if $q$ is odd,}\\
  \mathcal L(q;s;h+h_{q;s}) &\text{if $q$ is even and $\tau=\tau_h$.}
\end{cases}
$$
\end{lemma}

We call $\mathcal L_{\Gamma,\tau}$ the \emph{affine congruence lattice associated to $L$ and $\tau$}.

\begin{proof}
For any representation $\pi$ of $G$ one has the decomposition $V_\pi=\oplus_{\mu\in P(G)} V_\pi(\mu)$ in weight spaces, that is, an element $g\in T$ acts on each $V_\pi(\mu)$ by $g^\mu :=e^{\mu(X)}$, where $X\in\mathfrak h$ satisfies $\exp(X)=g$.
Since $\tau(\Gamma)\subset T$, we have that
$$
V_\pi^{\tau(\Gamma)}= \bigoplus_{\mu\in P(G)} V_\pi(\mu)^{\tau(\Gamma)}.
$$
Moreover, $\Gamma$ is cyclically generated by $\gamma$ as in \eqref{eq2:gamma}, thus a nonzero element $v\in V_\pi(\mu)$ is $\Gamma$-invariant if and only if $\tau(\gamma)^\mu=1$.
Hence
$$
\dim V_\pi^{\tau(\Gamma)} =\sum_{\mu} \dim V_\pi(\mu)^{\tau(\Gamma)} =\sum_{\mu} m_\pi(\mu)
$$
where $\mu$ runs through $P(G)$ satisfying $\tau(\gamma)^\mu=1$.

In the case at hand, $\pi=\pi_k^\pm$ for some integer $k\geq0$, we have that $m_{\pi_k^\pm}(\mu)=0$ for every $\mu\in\Z^m$ (see Lemma~\ref{lem4:multiplicity}).
It remains to show that $\{\mu\in \ZZ^m: \tau(\gamma)^\mu=1\}$ is equal to $\mathcal L(q;s)$ if $q$ is odd or equal to $\mathcal L(q;s;h)$ if $q$ is even and $\tau=\tau_h$.

We first assume $q$ odd.
Recall that $\tau$ is as in \eqref{eq2:spin-str-odd}
We let
$
X_{\tau(\gamma)}=
\sum_{j=1}^m \tfrac{(q+1)s_j\pi}{q}e_{2j-1}e_{2j},
$
thus $\exp (X_{\tau(\gamma)})=\tau(\gamma)$.
Then $\tau(\gamma)^\mu=e^{\mu(X_{\tau(\gamma)})}$ with
$$
\mu(X_{\tau(\gamma)})
= 2\pi i\frac{(q+1)}{2q}\sum_{j=1}^m a_js_j
$$
for $\mu=\tfrac12\sum_{j=1}^{m} a_j\varepsilon_j\in \ZZ^m$.
Hence, $\tau(\gamma)^\mu=1$ if and only if $\frac{q+1}2\sum_{j=1}^m a_js_j\in q\Z$.
Since $\frac{q+1}2$ is an integer coprime to $q$ and $\sum_{j=1}^m a_js_j\in\Z$, the assertion in this case follows.

We next assume $q$ even, thus $m$ is even and $\tau=\tau_h$ for some $h\in\{0,1\}$ as in \eqref{eq2:spin-str-even}.
Then
$$
\tau_h(\gamma)
= \prod_{j=1}^{m-1} \left(\cos(\tfrac{s_j\pi}{q}) + \sin(\tfrac{s_j\pi}{q})\, e_{2j-1}e_{2j}\right) \cdot \left(\cos(\tfrac{s_m+(h+h_{q;s})q}{q}\,\pi) + \sin(\tfrac{s_m+(h+h_{q;s})q}{q}\,\pi)\, e_{2m-1}e_{2m}\right).
$$
Similarly as above, we let
$
X_{\tau_h(\gamma)}=
\sum_{j=1}^{m-1} \tfrac{s_j\pi}{q}e_{2j-1}e_{2j} + \tfrac{s_m+(h+h_{q;s})q}{q} \pi e_{2m-1}e_{2m},
$
thus $\exp (X_{\tau_h(\gamma)})=\tau_h(\gamma)$, $\tau_h(\gamma)^\mu=e^{\mu(X_{\tau_h(\gamma)})}$ and
$$
\mu(X_{\tau_h(\gamma)})
= 2\pi i\left(\frac{1}{2q}\sum_{j=1}^m a_js_j + \frac{h+h_{q;s}}{2}\,a_m\right)
$$
for $\mu=\tfrac12\sum_{j=1}^{m} a_j\varepsilon_j\in \ZZ^m$.
Hence, $\tau_h(\gamma)^\mu=1$ if and only if $\sum_{j=1}^m a_js_j+a_m (h+h_{q;s})q\in 2q\Z$.
Since $a_m$ is an odd integer and $\sum_{j=1}^m a_js_j\in\Z$, the assertion follows.
\end{proof}

We next compute the multiplicities $m_{\pi_k^\pm}(\mu)$ to obtain a closed formula for $\dim V_{\pi_k^\pm}^{\tau(\Gamma)}$.
We recall from Section~\ref{sec:congafflattices} that $\norma{\mu}=\tfrac12\sum_{j=1}^m |a_j|$ and $\menos(\mu)=\#\{1\leq j\leq m: a_j<0\}$ if $\mu=\tfrac12(a_1,\dots,a_m)$.

\begin{lemma}\label{lem4:multiplicity}
Let $k$ be a non-negative integer, $\mu\in P(G)$ and $r:=\norma{\Lambda_k^\pm}-\norma{\mu}$.
If $\mu\in \Z^m$ then $m_{\pi_k^\pm}(\mu) = 0$  and, if $\mu\in\ZZ^m$ then
\begin{align*}
m_{\pi_k^+}(\mu) &=
\begin{cases}
\binom{r+m-2}{m-2} &\text{if } r\geq0\text{ and } \menos(\mu)\equiv r\pmod2,\\
0&\text{otherwise},
\end{cases}\\
m_{\pi_k^-}(\mu) &=
\begin{cases}
\binom{r+m-2}{m-2} &\text{if } r\geq0\text{ and } \menos(\mu)\equiv r+1\pmod2,\\
0&\text{otherwise}.
\end{cases}
\end{align*}
\end{lemma}

\begin{proof}
We have that $m_{\pi_k^\pm}(\mu)=0$ for every $\mu\in\Z^m$ since a weight $\mu$ of $\pi_{\Lambda_k^\pm}$ satisfies that $\Lambda_k^\pm-\mu$ is a sum of positive roots.
We now let $\mu\in\ZZ^m$.
We shall prove this case by induction on $k$.
Note that the assertion is clear for $k=0$ since the set of weights of $\pi_0^\pm$ is
\begin{equation}
\mathcal P_0^\pm:= \{\nu=\tfrac12\textstyle\sum_j b_j\varepsilon_j: |b_j|=1\;\forall j\text{ and } R(\nu)\equiv \tfrac{-1\pm1}{2} \pmod2\}
\end{equation}
and $m_{\pi_0^\pm}(\nu)=1$ for those weights.

We assume now that $k\in\N$.
By Steinberg's formula, one can prove that $\pi_{k\varepsilon_1}\otimes \pi_0^\pm$ decomposes into irreducible components as $\pi_k^\pm\oplus\pi_{k-1}^\mp$.
Then
\begin{align}\label{eq4:m_pi_k^pm}
m_{\pi_k^\pm}(\mu)
    &= m_{\pi_{k\varepsilon_1}\otimes \pi_0^\pm}(\mu)-m_{\pi_{k-1}^\mp}(\mu) \\
    &= \sum_{\mu_1+\mu_2=\mu} m_{\pi_{k\varepsilon_1}}(\mu_1) \, m_{\pi_0^\pm}(\mu_2) - m_{\pi_{k-1}^\mp}(\mu)
    \notag\\
    &= \sum_{\nu\in\mathcal P_0^\pm} m_{\pi_{k\varepsilon_1}}(\mu-\nu) - m_{\pi_{k-1}^\mp}(\mu).
    \notag
\end{align}
Furthermore, one has that
\begin{equation}\label{eq4:m_pi_ke_1}
m_{\pi_{k\varepsilon_1}} (\eta) =
\begin{cases}
  \binom{l+m-2}{m-2}\quad&\text{if } \norma{\eta}=k-2l\text{ with }l\in\N_0,\\
  0\quad&\text{otherwise,}
\end{cases}
\end{equation}
for every $\eta\in\Z^m$ (see for instance \cite[Lem.~3.5]{LMR}).

We write $\mu=\tfrac12\sum_i a_i\varepsilon_i\in\ZZ^m$ as $\mu=\eta_0+\nu_0$ where $\nu_0=\tfrac12\sum_i \sgn{a_i} \varepsilon_i \in\mathcal P_0^+\cup\mathcal P_0^-$ ($\sgn{a_i}:=\tfrac{a_i}{|a_i|}$) and $\eta_0=\mu-\nu_0=\tfrac12\sum_i(a_i-\sgn{a_i})\varepsilon_i\in\Z^m$.
Note that all the coordinates of $\mu$ and $\nu_0$ have the same sign, $\norma{\mu}=\norma{\nu_0}+\norma{\eta_0}=\tfrac m2+\norma{\eta_0}$ and, $\nu_0$ is in $\mathcal P_0^+$ or $\mathcal P_0^-$ according to the parity of $R(\mu)=R(\nu_0)$.

Let $r=\norma{\Lambda_k^\pm}-\norma{\mu}\in\Z$, thus $r=k-\norma{\eta_0}$.
For $\nu\in \mathcal P_0^+\cup\mathcal P_0^-$, we have that
\begin{equation}\label{eq4:||mu-nu||_1}
\norma{\mu-\nu} = \norma{\eta_0}+\norma{\nu_0-\nu}= k-r+\norma{\nu_0-\nu}.
\end{equation}
If $r<0$ then $\norma{\mu-\nu}>k$, which implies that $m_{\pi_{k\varepsilon_1}}(\mu-\nu)=0$ for every $\nu\in \mathcal P_0^+\cup\mathcal P_0^-$ by \eqref{eq4:m_pi_ke_1}.
Furthermore, $m_{\pi_{k-1}^\mp}(\mu)=0$ by hypothesis.
Hence, $m_{\pi_{k}^\pm}(\mu)=0$ by \eqref{eq4:m_pi_k^pm} if $r<0$.

Suppose $r\geq0$.
By \eqref{eq4:||mu-nu||_1} and \eqref{eq4:m_pi_ke_1}, the sum at the right hand side of \eqref{eq4:m_pi_k^pm} is reduced to the elements $\nu\in\mathcal P_0^\pm$ such that
\begin{equation}\label{eq4:nu-condition}
r-\norma{\nu_0-\nu}\text{ is nonnegative and even.}
\end{equation}

We next prove the assertion in the case when $r$ and $R(\mu)$ are both even.
The other cases are very similar.
Note that $\norma{\nu_0-\nu}$ is the number of coordinates in which $\nu_0$ and $\nu$ differ.
Then $\norma{\nu_0-\nu}$ and $r-\norma{\nu_0-\nu}$ are odd for any $\nu\in\mathcal P_0^-$ since $R(\nu)$ and $R(\nu_0)$ have different parity.
Furthermore, $m_{\pi_{k-1}^+}(\mu)=0$ by hypothesis.
Hence $m_{\pi_{k}^-}(\mu)=0$.

For $0\leq l\leq r/2$, there are $\binom{m}{2l}$ elements in $\mathcal P_0^+$ coinciding in $2l$ coordinates with $\nu_0$.
For such $\nu$ we have that $m_{\pi_{k\varepsilon_1}}(\mu-\nu)=\binom{r/2-l+m-2}{m-2}$ by \eqref{eq4:m_pi_ke_1}.
Furthermore, $m_{\pi_{k-1}^-}(\mu)=\binom{r-1+m-2}{m-2}$ by hypothesis.
Hence, by \eqref{eq4:m_pi_k^pm}, we obtain that
\begin{equation}\label{eq4:mult-comb}
m_{\pi_k^+}(\mu) = \sum_{l=0}^{r/2} \binom{m}{2l}\,\binom{r/2-l+m-2}{m-2}- \binom{r-1+m-2}{m-2}.
\end{equation}

To prove $m_{\pi_k^+}(\mu) = \binom{r+m-2}{m-2}$,
from \eqref{eq4:mult-comb}, it is sufficient to show the identity
\begin{equation}\label{eq4:mult-comb}
\binom{r+m-2}{m-2}+\binom{r-1+m-2}{m-2} = \sum_{l=0}^{r/2} \binom{m}{2l}\,\binom{r/2-l+m-2}{m-2}.
\end{equation}
One can check that the left hand side is the $r$-th term of the series $F_0(x):=(1+x)/(1-x)^{m-1}$.
Moreover, since we are assuming that $r$ is even, it is also the $r$-th term of the even part of $F_0(x)$, namely, $F_1(x):=\tfrac12(F_0(x)+F_0(-x))$.
We have that
\begin{align*}
F_1(x)
    &= \frac12\left(\frac{1+x}{(1-x)^{m-1}} + \frac{1-x}{(1+x)^{m-1}}\right) \\
    &= \frac{(1+x)^{m} + (1-x)^{m}}{2}\;\frac{1}{(1-x^2)^{m-1}}   \\
    &= \left(\sum_{l=0}^{\lfloor m/2\rfloor} \binom{m}{2l} \; x^{2l}\right) \left(\sum_{h\geq0} \binom{h+m-2}{m-2} \; x^{2h}\right) \\
    &= \sum_{h\geq0\atop h \text{ even}} \left(\sum_{l=0}^{\lfloor m/2\rfloor} \binom{m}{2l} \binom{h/2-l+m-2}{m-2}\right)x^h.
\end{align*}
This implies \eqref{eq4:mult-comb}, which completes the proof.
\end{proof}

We are now in a condition to state the main theorem in this paper, namely, the description of the Dirac spectrum of a lens space endowed with a spin structure.

We recall from Section~\ref{sec:congafflattices} that, for any $\epsilon\in\Z$, $\mathcal L_{\epsilon}$ denotes the subset of $\mathcal L$ given by elements $\mu$ satisfying $R(\mu)\equiv\epsilon\pmod2$.
To facilitate the reading, we introduce more notation.
For $r\geq0$ and $\epsilon\in\Z$, we let
\begin{align}\label{eq4:N_L(epsilon,r)}
\mathcal L_{\epsilon,r}
    &= \{\mu\in\mathcal L_\epsilon: \norma{\mu}=r+\tfrac m2\},\\
N_{\mathcal L}(\epsilon,r)
    &= \#\mathcal L_{\epsilon,r}.\notag
\end{align}

\begin{theorem}\label{thm4:dimV-lens}
Let $L=\Gamma\ba S^{2m-1}$ be a lens space with spin structure $\tau$ and let $\mathcal L=\mathcal L_{\Gamma,\tau}$ be its associated affine congruence lattice.
For $k\geq0$, let $\lambda_k=k+\tfrac{2m-1}{2}$.
Then, the eigenvalues of the Dirac operator on $L$ are  $\pm\lambda_k$ with multiplicity
\begin{align*}
\mult_{(L,\tau)}(-\lambda_k)
    &= \sum_{r=0}^k \tbinom{r+m-2}{m-2} \;N_{\mathcal L} (r,k-r),\\
\mult_{(L,\tau)}(+\lambda_k)
    &= \sum_{r=0}^k \tbinom{r+m-2}{m-2} \;N_{\mathcal L} (r+1,k-r).
\end{align*}
\end{theorem}

\begin{proof}
From Proposition~\ref{prop2:spectrum} and Lemma~\ref{lem4:dimV-lens} we have that
$$
\mult_{(L,\tau)}(\mp\lambda_k) = \dim V_{\pi_k^\pm}^{\tau(\Gamma)} = \sum_{\mu\in\mathcal L} m_{\pi_k^\pm}(\mu).
$$
By Lemma~\ref{lem4:multiplicity}, every weight $\mu$ of $\pi_k^\pm$ (i.e.\ $m_{\pi_k^\pm}(\mu)>0$) is in $\ZZ^m$ and satisfies $\norma{\mu}=\norma{\Lambda_k^\pm}-r=(k-r)+\frac m2$ for some non-negative integer $r\leq k$.
Hence,
\begin{align*}
\mult_{(L,\tau)}(-\lambda_k)
    &= \sum_{r=0}^k\sum_{\epsilon=0}^1 \sum_{\mu\in\mathcal L_{r+\epsilon,k-r}} m_{\pi_k^+}(\mu)
    = \sum_{r=0}^k \sum_{\mu\in\mathcal L_{r,k-r}} \tbinom{r+m-2}{m-2},
\end{align*}
which establishes the formulas.
The case $\mult_{(L,\tau)}(+\lambda_k)$ is very similar.
\end{proof}

\begin{definition}
Two subsets $\mathcal L$ and $\mathcal L'$ of $\ZZ^m$ are \emph{oriented $\norma{\cdot}$-isospectral} if $N_{\mathcal L}(\epsilon,k) = N_{\mathcal L'}(\epsilon,k)$ for every non-negative integer $k$ and every $\epsilon=0,1$.
Similarly, we say they are \emph{$\norma{\cdot}$-isospectral} if $N_{\mathcal L}(0,k) + N_{\mathcal L}(1,k) = N_{\mathcal L'}(0,k) + N_{\mathcal L'}(1,k)$ for every $k\geq0$.
\end{definition}

\begin{corollary}\label{cor4:char-lens}
Let $L$ and $L'$ be lens spaces with spin structures $\tau,\tau'$ and associated affine congruence lattices $\mathcal L$ and $\mathcal L'$ respectively.
Then, $L$ and $L'$ are Dirac isospectral if and only if $\mathcal L$ and $\mathcal L'$ are oriented $\norma{\cdot}$-isospectral.
\end{corollary}

\begin{proof}
Proposition~\ref{prop2:spectrum} and Theorem~\ref{thm4:dimV-lens} imply immediately that $L$ and $L'$ are Dirac isospectral if $\mathcal L$ and $\mathcal L'$ are oriented $\norma{\cdot}$-isospectral.
We now assume that $L$ and $L'$ are Dirac isospectral, thus $\dim V_{\pi_k^\pm}^{\tau(\Gamma)}=\dim V_{\pi_k^\pm}^{\tau(\Gamma')}$ for every non-negative integer $k$.
Write $\epsilon^+=0$ and $\epsilon^-=1$.
We shall prove by induction on $k$ that
\begin{equation}\label{eq4:N(k,e)=N'(k,e)}
N_{\mathcal L}(\epsilon^\pm,k) = N_{\mathcal L'}(\epsilon^\pm,k)
\end{equation}
for every $k \geq0$.
Theorem~\ref{thm4:dimV-lens} implies that $N_{\mathcal L} (\epsilon^\pm,0) = \dim V_{\pi_0^\pm}^{\tau(\Gamma)}=\dim V_{\pi_0^\pm}^{\tau(\Gamma')} = N_{\mathcal L'} (\epsilon^\pm,0)$, thus the case $k=0$ is proved.
Suppose that \eqref{eq4:N(k,e)=N'(k,e)} holds for every $k<k_0$.
By Theorem~\ref{thm4:dimV-lens} we have that
$$
\sum_{r=0}^{k_0} \tbinom{r+m-2}{m-2} \;N_{\mathcal L} (r+\epsilon^\pm,k-r)
=
\sum_{r=0}^{k_0} \tbinom{r+m-2}{m-2} \;N_{\mathcal L'} (r+\epsilon^\pm,k-r).
$$
All the terms with $r>0$ on both sides are equal by assumption, hence this equality implies that $N_{\mathcal L} (\epsilon^\pm,k)=N_{\mathcal L'} (\epsilon^\pm,k)$, which completes the proof.
\end{proof}

\begin{corollary}
Under the same assumptions as in Corollary~\ref{cor4:char-lens}, when $m$ is odd (thus $q$ is odd), $L$ and $L'$ are Dirac isospectral if and only if $\mathcal L$ and $\mathcal L'$ are $\norma{\cdot}$-isospectral.
\end{corollary}

\begin{proof}
Since $m$ is odd, we have that $R(\mu)\equiv R(-\mu)+1\pmod2$ for every $\mu\in\ZZ^m$.
Hence $m_{\pi_k^+}(-\mu) = m_{\pi_k^-}(\mu)$ by Lemma~\ref{lem4:multiplicity}.
This implies that $N_{\mathcal L}(0,r) = N_{\mathcal L}(1,r)$ since $\mu\in\mathcal L_{\epsilon,r}$ if and only if $-\mu\in\mathcal L_{\epsilon+1,r}$.
Hence, oriented $\norma{\cdot}$-isospectrality coincides with $\norma{\cdot}$-isospectrality.
\end{proof}

\begin{remark}
It is also reasonable to define \emph{reversed $\norma{\cdot}$-isospectrality} as $N_{\mathcal L}(\epsilon,k)=N_{\mathcal L'}(\epsilon+1,k)$ for every $\epsilon=0,1$.
Then, under the same hypotheses as in Corollary~\ref{cor4:char-lens}, $\mathcal L$ and $\mathcal L'$ are reversed oriented $\norma{\cdot}$-isospectral if and only if $L$ and $L'$ are \emph{inverse Dirac isospectral}, that is, $F_{(L,\tau)}^\pm(z) = F_{(L',\tau')}^\mp(z)$ (see \cite[Prop.~2.15]{Boldt}).
However, we can change the (fixed) orientation in one of the lens spaces to obtain the usual Dirac isospectrality.

Here is another option which works with the already fixed orientation on a lens space.
We multiply exactly one parameter of $L'$ by $-1$, obtaining an isometric new lens space $L''$ (note that the isometry between $L'$ and $L''$ reverses orientations).
Then, we have that $F_{(L,\tau)}^\pm(z) = F_{(L'',\tau'')}^\pm(z)$.
Hence, $L$ and $L''$ are Dirac isospectral.
\end{remark}

\section{Infinite families of Dirac isospectral lens spaces}\label{sec:families}
The goal of this section is to construct examples of Dirac isospectral lens spaces by using Corollary~\ref{cor4:char-lens}.
More precisely, we present three infinite families of Dirac isospectral lens spaces and thereby show that neither spin structures nor isometry classes of lens spaces are spectrally determined.

\begin{remark}\label{rem5:trivial-isosp}
There are some ``trivial'' occurrences of Dirac isospectrality given by lens spaces with two different spin structures related by an isometry.
Let us have a look at an example.
Let $L:=L(16;1,3,5,7)$, which admits two different spin structures $\tau_0$ and $\tau_1$ by Proposition~\ref{thm2:spin-structures}.
The function $\Psi:S^7\to S^7$, $(z_1,z_2,z_3,z_4)\mapsto (z_3,\overline z_1,\overline z_4,z_2)$ induces an isometry on $L$ that takes $\tau_0$ to $\tau_1$.
Indeed, this follows from Proposition~\ref{prop2:lens-isom-explicit} by choosing the permutation $\sigma=(1 3)(3 4)(4 2)$, signs $\epsilon_2=\epsilon_3=-\epsilon_1=-\epsilon_4=1$ and $\ell=11$.
\end{remark}

We start by giving a sequence of families of Dirac isospectral lens spaces with increasing cardinality and dimension and a fixed order of the fundamental group.

\begin{theorem}\label{thm5:increasing}
Let $q=40$.
For each $r\geq1$ we set $m=4r+2$
and, for each $0\leq p\leq \tfrac{m-2}{4}=r$ we let
\begin{equation*}
s^{(p)}=(\underbrace{1,11,\dots,1,11}_{m-2p}, \underbrace{21,31,\dots,21,31}_{2p}).
\end{equation*}
Then, the lens spaces in the family $\{L(q,s^{(p)}):0\leq p\leq r\}$ are pairwise non-isometric and,
when each is endowed with the spin structure $\tau_0^{(p)}$ (see \eqref{eq2:spin-str-even}), they are pairwise Dirac isospectral.
The dimension of the lens spaces is $8r+3$ and the cardinality of the family is $r+1$.
\end{theorem}

\begin{proof}
Throughout the proof, we fix the positive integer $r$.
We will show that $(L(q,s^{(0)}),\tau_0^{(0)})$ and $(L(q,s^{(p)}),\tau_0^{(p)})$ are Dirac isospectral for each $1\leq p\leq r$.
From Corollary~\ref{cor4:char-lens}, this is equivalent to proving that the associated affine congruence lattices $\mathcal L_0:=\mathcal L(q;s^{(0)};0)$ and $\mathcal L_p:=\mathcal L(q;s^{(p)};0)$ are oriented $\norma{\cdot}$-isospectral.

Let $\mu=\tfrac12(a_1,b_1,\dots,a_{\tfrac{m}{2}},b_{\tfrac{m}{2}})\in\ZZ^m$.
Then
\begin{align*}
\mu&\in \mathcal L_0 &\Longleftrightarrow &&
\sum_{j=1}^{m/2} (a_j+11b_j)&\equiv 0\pmod{80},\\
\mu&\in \mathcal L_p &\Longleftrightarrow &&
20\sum_{j=\tfrac m2-p+1}^{p} (a_j+b_j) +\sum_{j=1}^{m/2} (a_j+11b_j)&\equiv 0\pmod{80}.\\
\end{align*}
We consider the map
\begin{equation*}
\tfrac12(a_1,b_1,\dots,a_{\tfrac{m}{2}},b_{\tfrac{m}{2}}) \longmapsto
\begin{cases}
\tfrac12(a_1,b_1,\dots,a_{\tfrac{m}{2}},b_{\tfrac{m}{2}})
    &\quad\text{if }\displaystyle\sum_{j=\tfrac m2-p+1}^{p} (a_j+b_j)\equiv0\pmod4,\\
\tfrac12(b_1,a_1,\dots,b_{\tfrac{m}{2}},a_{\tfrac{m}{2}})
    &\quad\text{if }\displaystyle\sum_{j=\tfrac m2-p+1}^{p} (a_j+b_j)\equiv2\pmod4.
\end{cases}
\end{equation*}
It follows that this map gives a bijection between $\mathcal L_0$ and $\mathcal L_p$.
The second row in the map works as follows.
If $\tfrac12(a_1,b_1,\dots,a_{\tfrac{m}{2}},b_{\tfrac{m}{2}})\in\mathcal L_0$, then $0\equiv11\sum_{j=1}^{m/2} (a_j+11b_j)\equiv \sum_{j=1}^{m/2} (b_j+11a_j)+40(\sum_{j=1}^{m/2} b_j) \equiv \sum_{j=1}^{m/2} (b_j+11a_j)+40\pmod{80}$ since $m/2$ is odd, hence $\tfrac12(b_1,a_1,\dots,b_{\tfrac{m}{2}},a_{\tfrac{m}{2}})\in\mathcal L_p$ since $20\sum_{j=\tfrac m2-p+1}^{p} (b_j+a_j) +\sum_{j=1}^{m/2} (b_j+11a_j) \equiv 40+40\equiv 0\pmod{80}$.
This proves that $\mathcal L_0$ and $\mathcal L_p$ are oriented $\norma{\cdot}$-isospectral since the bijection preserves $\norma{\cdot}$ and $R(\cdot)$.

The non-isometry between the lens spaces follows from Proposition~\ref{prop2:lens-isom}.
\end{proof}

\begin{remark}
It is curious that for every $0\leq p\leq \tfrac{m-2}{4}$, one can check that $\mathcal L(40;s^{(0)})=\mathcal L(40;s^{(p)})$ (see \eqref{eq3:L(q;s)}) though the lens spaces $L(40;s^{(0)})$ and $L(40;s^{(p)})$ are not isometric.
Furthermore, $\mathcal L(40;s^{(p)};0)$ and $\mathcal L(40;s^{(p)};1)$ are $\norma{\cdot}$-isometric thus there is an isometry of the lens space $L(40;s^{(p)})$ that switches the spin structures $\tau_0^{(p)}$ and $\tau_1^{(p)}$.
\end{remark}

We continue with a family of pairs of Dirac isospectral lens spaces that have the same underlying manifold but different spin structures (see Subsection~\ref{subsec:lens}).

\begin{theorem}\label{thm5:q-even}
For any $\rr\geq1$, we consider the lens space $$L=L(32\rr; 1,1+4\rr,1+16\rr,1+28\rr).$$
Then, $L$ does not carry an isometry which takes the spin structure $\tau_0$ to $\tau_1$, but is Dirac isospectral to itself when equipped with the two different spin structures.
\end{theorem}

\begin{proof}
By Proposition~\ref{prop2:lens-isom-explicit}, to check that there is no isometry relating the two spin structures, it is sufficient to show that multiplying the parameter tuple $s:=(1,1+4\rr ,1+16\rr ,1+28\rr )$ by the inverses modulo $q:=32\rr$ of its entries does not produce, up to sign in each entry, a permutation of the parameter tuple $s$.
This is easily seen to be the case since the inverses of $1$, $1+4\rr$, $1+16\rr$, $1+28\rr$ are $1$, $1+4\rr (2^\alpha4-1)$, $1+16\rr$, $1+4\rr(2^\alpha4+1)$ respectively, where $\rr=2^{\alpha}\rr_1$ with $\rr_1$ odd.

We now prove Dirac isospectrality by showing that the affine congruence lattices $\mathcal L^{(0)}:=\mathcal L(q;s;0)$ and $\mathcal L^{(1)}:=\mathcal L(q;s;1)$ are oriented $\norma{\cdot}$-isospectral, which is sufficient by Corollary~\ref{cor4:char-lens}.
By definition (see \eqref{eq3:L(q;s;h)}), $\mu=\tfrac12(a,b,c,d)\in\mathcal L^{(h)}$ if and only if $a,b,c,d$ are odd integer numbers such that
$$
a+b(1+4\rr)+c(1+16\rr)+d(1+28\rr) \equiv 32\rr h\pmod{64\rr},
$$
or equivalently,
\begin{equation*}
(a+b+c+d)+8\rr\left(\frac{b+7d}{2}+2c\right) \equiv 32\rr h\pmod{64\rr}.
\end{equation*}

Note that if $\mu=\tfrac12(a,b,c,d)\in\mathcal L:=\mathcal L^{(0)}\cup \mathcal L^{(1)}$, then $a+b+c+d\equiv 0\pmod{8\rr}$.
For $\mu\in\ZZ^4$ we let $\gamma_\mu=(a+b+c+d)/8\rr$.
Now, $\mu\in\ZZ^4$ is in $\mathcal L^{(h)}$ if and only if $\gamma_\mu\in\Z$ and
\begin{equation}\label{eq5:L_qsh_q=even2}
\gamma_\mu+\frac{b-d}{2}+2c+4 \equiv 4h\pmod{8}
\end{equation}
We will prove that
\begin{equation}\label{eq5:card(L_qsh)}
\#\{\mu\in \mathcal L^{(0)}_{\epsilon,k}: \gamma_\mu=\gammab \}=\#\{\mu\in \mathcal L^{(1)}_{\epsilon,k}: \gamma_\mu=\gammab \}
\end{equation}
for every $\gammab\in\Z$, $k\geq0$ and $\epsilon=0,1$.
The next lemma is elementary and will be used many times in the sequel.

\begin{lemma}\label{lem5:(b-d)mod4}
Let $\mu=\tfrac12(a,b,c,d)\in\mathcal L$.
Then, $b-d\equiv2\pmod 4$ $\iff$ $b+d\equiv0\pmod 4$ $\iff$ $a+c\equiv0\pmod 4$ $\iff$ $a-c\equiv2\pmod 4$, and also $b-d\equiv0\pmod 4$ $\iff$ $b+d\equiv2\pmod 4$ $\iff$ $a+c\equiv2\pmod 4$ $\iff$ $a-c\equiv0\pmod 4$.
\end{lemma}
\begin{proof}
We have that $a+b+c+d\equiv0\pmod 4$ since $\mu\in\mathcal L$.
The claim follows immediately since $a,b,c,d$ are odd integers.
\end{proof}

We proceed with the proof of \eqref{eq5:card(L_qsh)} and split it into three cases.
First assume $\gammab\equiv0\pmod 4$.
Define
\begin{equation*}
\Psi_1:\ZZ^4\longrightarrow\ZZ^4, \qquad \tfrac12(a,b,c,d)\longmapsto \tfrac12(c,d,a,b).
\end{equation*}
We will show that $\Psi_1$ induces a bijection between the sets in \eqref{eq5:card(L_qsh)}.
Clearly, $\Psi_1$ preserves the functions $\norma{\cdot}$, $R(\cdot)$ and $\gamma_{(\cdot)}$, thus it is sufficient to show that $\mu\in \mathcal L^{(0)}$ if and only if $\Psi_1(\mu)\in\mathcal L^{(1)}$ for any $\mu\in\ZZ^4$ satisfying $\gamma_\mu= \gammab$.

Let $\mu'=\tfrac12(a',b',c',d')=\Psi_1(\mu)$, where $\mu\in\mathcal L^{(0)}$ with $\gamma_\mu=\gammab$ divisible by $4$.
We check that $\mu'\in\mathcal L^{(1)}$ by using the condition \eqref{eq5:L_qsh_q=even2}.
We have that
\begin{align*}
\gamma_{\mu'} + \tfrac{b'-d'}{2} +2c'+4
  &\equiv \gamma_{\mu} + \tfrac{d-b}{2} +2a+4 \pmod{8},\\
  &\equiv \gamma_{\mu} + \tfrac{b-d}{2} +2c+4 +(d-b)+2(a-c) \pmod{8},\\
  &\equiv (d-b)+2(a-c) \pmod{8}.
\end{align*}
On the other hand, since $\mu\in\mathcal L^{(0)}$, \eqref{eq5:L_qsh_q=even2} implies that $\tfrac{b-d}{2}\equiv2\pmod4$ because $\gamma_\mu$ is divisible by $4$, thus $d-b\equiv4\pmod8$.
In particular $b-d$ is divisible by $4$, then $a-c\equiv 0\pmod4$ by Lemma~\ref{lem5:(b-d)mod4}.
Hence $\gamma_{\mu'} + \tfrac{b'-d'}{2} +2c'+4 \equiv 4\pmod8$, that is, $\Psi_1(\mu)\in \mathcal L^{(1)}$.
The converse is very similar.

We now consider the case $\gammab\equiv1\pmod2$.
We define
\begin{equation*}
\Psi_2:\ZZ^4\longrightarrow \ZZ^4, \qquad \tfrac12(a,b,c,d)\longmapsto \tfrac12(c,b,a,d).
\end{equation*}
Again, $\Psi_2$ preserves $\norma{\cdot}$, $R(\cdot)$ and $\gamma_{(\cdot)}$.
Let $\mu=\tfrac12(a,b,c,d)\in \mathcal L^{(0)}$ such that $\gamma_\mu=\gammab$, and write $\mu'=\tfrac12(a',b',c',d')=\Psi_2(\mu)$.
One can check that
\begin{align*}
\gamma_{\mu'} + \tfrac{b'-d'}{2} +2c'+4
  &\equiv 2(a-c) \pmod{8}.
\end{align*}
By \eqref{eq5:L_qsh_q=even2}, $(b-d)/2\equiv-\gamma_\mu\equiv 1\pmod2$, then $b-d\equiv 2\pmod4$, thus $a-c\equiv 2\pmod4$ by Lemma~\ref{lem5:(b-d)mod4}.
Hence $\mu'\in \mathcal L^{(1)}$.

We now assume $\gammab\equiv2\pmod4$.
This case is the most involved since there is no $\norma{\cdot}$-preserving bijection between the sets in \eqref{eq5:card(L_qsh)}.
We will check \eqref{eq5:card(L_qsh)} in each of the $16$ orthants of $\ZZ^4$.
Actually, since $\mu\in\mathcal L^{(h)}$ if and only if $-\mu\in\mathcal L^{(h)}$, it is sufficient to check half of the orthants.
We will refer to orthants by the signs of the coordinates of their elements.
For example, ${+}{-}{+}{-}$ refers to the orthant of $\ZZ^4$ given by elements $\tfrac12(a,b,c,d)\in\ZZ^4$ satisfying $a,c>0$ and $b,d<0$.

We start with the orthant ${+}{+}{+}{+}$ of $\ZZ^4$, that is, the orthant where all the coordinates are positive.
Every element in this orthant with one-norm $k$ can be written as
\begin{equation}\label{eq5:q-even-mu++++}
\mu=\tfrac12(a,b,c,d)=\tfrac12\big(\, 2(k-2-x)+1, \,2(x-y)+1, \,2(y-z)+1, \,2z+1\big),
\end{equation}
with $0\leq z\leq y\leq x\leq k-2$ integer numbers.
Note that $\gamma_\mu = 2\norma{\mu}/8r$, thus we will assume $k=4r\gammab$.
By \eqref{eq5:L_qsh_q=even2}, $\mu\in\mathcal L^{(h)}$ if and only if
\begin{align*}
 \gammab+6 +x +3y-5z \equiv 4h\pmod8.
\end{align*}
We assume that $\gammab\equiv 2\pmod 8$, while the case $\gammab\equiv 6\pmod 8$ is analogous.
Set $\xi=e^{2\pi i/8}$, a primitive $8$-th root of unity.
In the following we often use the fact that $\sum_{j=1}^{n} w^j$ is equal to $\tfrac{w^{n+1}-1}{w-1}$ if $w\neq1$ and equal to $n$ if $w=1$.
In particular $\tfrac18\sum_{l=0}^{7} \xi^{jl}$ is equal to $0$ if $j\not\equiv 0\pmod 8$ and equal to $1$ if $j\equiv0\pmod 8$.

Clearly, the number $N_{\mathcal L^{(h)},\gammab}^{{+}{+}{+}{+}}(k)$ of elements $\mu\in\mathcal L^{(h)}$ in the orthant ${+}{+}{+}{+}$ with $\gamma_\mu=\gammab$ and $\norma{\mu}=k$ is equal to
\begin{align}\label{eq5:q-even-N++++1}
N_{\mathcal L^{(h)},\gammab}^{{+}{+}{+}{+}}(k)
&= \sum_{x=0}^{k-2} \sum_{y=0}^{x} \sum_{z=0}^{y} \frac18 \sum_{l=0}^{7} \xi^{l(x +3y-5z-4h)}
= \frac18 \sum_{l=0}^{7} (-1)^{hl} \sum_{x=0}^{k-2} \xi^{lx}\sum_{y=0}^{x} \xi^{3ly} \sum_{z=0}^{y} \xi^{3lz}.
\end{align}
Our goal is to prove that $N_{\mathcal L^{(0)},\gammab}^{{+}{+}{+}{+}}(k)=N_{\mathcal L^{(1)},\gammab}^{{+}{+}{+}{+}}(k)$, or equivalently, that $N_{\mathcal L^{(h)},\gammab}^{{+}{+}{+}{+}}(k)$ does not depend on $h$.
It is clear that for every even $\ell$ the terms in \eqref{eq5:q-even-N++++1} for $h=0$ and for $h=1$ coincide.
It is thus sufficient to show that any term with $l$ odd vanishes.
We suppose $l\in\{1,3,5,7\}$, then
\begin{align*}
A_l
:=&\; \frac{(-1)^{lh}}{8} \sum_{x=0}^{k-2} \xi^{lx}\sum_{y=0}^{x} \xi^{3ly}  \sum_{z=0}^{y} \xi^{3lz}
 =\; \frac{(-1)^{h}}{8} \sum_{x=0}^{k-2} \xi^{lx}\sum_{y=0}^{x} \xi^{3ly}  \,\frac{(\xi^{3l})^{y+1}-1}{\xi^{3l}-1}.
\end{align*}
Note that we have used that $\xi^{3l}\neq1$ in the last step.
By straightforward computations we obtain that
\begin{align*}
A_l
&= \frac{(-1)^{h}\, \xi^{3l}}{(\xi^{3l}-1)(\xi^{6l}-1)} \sum_{x=0}^{k-2} \big(\xi^{6l} \xi^{7lx}-\xi^{lx}\big)
-\frac{(-1)^{h}}{(\xi^{3l}-1)^2} \sum_{x=0}^{k-2}  \big(\xi^{3l} (-1)^{x}-\xi^{lx}\big).
\end{align*}
Since $k=4r\gammab$ is divisible by $8$, it follows that $\sum_{x=0}^{k-2} \big(\xi^{6l} \xi^{7lx}-\xi^{lx}\big) = -\xi^{6l} \xi^{-7l}+\xi^{-l}=0$ and $\sum_{x=0}^{k-2}  \big(\xi^{3l} (-1)^{x}-\xi^{lx}\big)= \xi^{3l} +\xi^{-l}=0$.
Hence $A_l=0$.
This concludes the proof for the case ${+}{+}{+}{+}$.

The same method we used for the orthant ${+}{+}{+}{+}$ can be applied in the remaining cases.
These are even simpler than in ${+}{+}{+}{+}$ since one obtains two free parameters instead of three free parameters.
As an example, we consider the orthant ${+}{+}{-}{-}$.

Any element in the orthant ${+}{+}{-}{-}$ whose one-norm equals $k$ is of the form
\begin{equation}\label{eq5:q-even-mu++--}
\mu=\tfrac12(a,b,c,d)=\tfrac12\big(\, 2(k-2-x)+1, \,2(x-y)+1, \,-2(y-z)-1, \,-2z-1\big),
\end{equation}
with $0\leq z\leq y\leq x\leq k-2$.
Since $a+b+c+d= 2(k-2)-4y$, we have that $8r\gammab =2(k-2-2y)$ if $\gamma_\mu=\gammab$, thus $k$ is even and $y=k/2-1-2r\gammab$.
This fact explains that this case has one parameter less than the previous one.
We will assume that $0\leq y\leq k-2$, otherwise there are no elements $\mu$ in the orthant ${+}{+}{-}{-}$ such that $\norma{\mu}=k$ and $\gamma_\mu=\gammab$.

By \eqref{eq5:L_qsh_q=even2}, we obtain that $\mu\in\mathcal L^{(h)}$ if and only if
$$
\gammab+3+x-5y+5z\equiv 4h\pmod8.
$$
Hence
\begin{align*}\label{eq5:q-even-N++--1}
N_{\mathcal L^{(h)},\gammab}^{{+}{+}{-}{-}}(k)
&= \sum_{x=0}^{k-2} \sum_{z=0}^{y} \frac18 \sum_{l=0}^{7} \xi^{l(\gammab+3+x-5y+5z-4h)}
= \frac18 \sum_{l=0}^{7} (-1)^{hl} \xi^{l(\gammab+3-5y)} \sum_{x=y}^{k-2} \xi^{lx} \sum_{z=0}^{y} \xi^{5lz}.
\end{align*}
Similarly to the previous case, we want to show that $N_{\mathcal L^{(h)},\gammab}^{{+}{+}{-}{-}}(k)$ does not depend on $h$, thus it suffices to prove that the sum over $l=1,3,5,7$ vanishes.
Namely
\begin{align*}
B
:=&\, \frac18 \sum_{l=0\atop l\text{ odd}}^{7} (-1)^{hl} \xi^{l(\gammab+3-5y)} \sum_{x=y}^{2y+4r\gammab} \xi^{lx} \sum_{z=0}^{y} \xi^{5lz} \\
 =&\, \frac{(-1)^{h}}8 \sum_{l=0\atop l\text{ odd}}^{7}  \frac{\xi^{l(\gammab+3)}}{\xi^{2l}-1}\,\big(\xi^{6l(y+1)}+1-\xi^{l(y+1)}-\xi^{5l(y+1)}\big).
\end{align*}
Now, if $y$ is odd then $B=0$ since the terms for $l$ and $l+4$ have opposite signs.
One can check that $B=0$ for any even $y$ by considering the cases $y\equiv1,3,5,7\pmod 8$ separately.
This finishes the proof for the case ${+}{+}{-}{-}$.
The remaining cases are very similar.
\end{proof}

We next present an infinite family of pairs of non-isometric lens spaces that are Dirac isospectral.
In this case, the order $q$ of the fundamental group is odd and thus the lens spaces admit exactly one spin structure.
The next proof was inspired by that one in \cite[Thm.~6.3]{LMR}.
It gives formulas for $N_{\mathcal L}(\epsilon,k)$ and $N_{\mathcal L'}(\epsilon,k)$ showing that they coincide.

\begin{theorem}\label{thm5:q-odd}
For any odd positive integer $r\geq7$, we consider the lens spaces
\begin{eqnarray*}
  L &=& L(r^2; 1,1+r, 1+2r, 1+4r), \\
  L'&=& L(r^2; 1,1-r, 1-2r, 1-4r),
\end{eqnarray*}
endowed with their unique spin structure ($q:=r^2$ is odd).
Then, $L$ and $L'$ are non-isometric and Dirac isospectral.
\end{theorem}

\begin{proof}
One can check by using Proposition~\ref{prop2:lens-isom} that $L$ and $L'$ are not isometric.
We will prove Dirac isospectrality by showing that the congruence sets $\mathcal L:=\mathcal L(r^2;1,1+r, 1+2r, 1+4r)$ and $\mathcal L':=\mathcal L(r^2;1,1-r, 1-2r, 1-4r)$ are oriented $\norma{\cdot}$-isospectral by Corollary~\ref{cor4:char-lens}.
To do this, we will check for fixed $k$ that there are the same number of elements in $\mathcal L$ and in $\mathcal L'$ with one-norm equal to $k$ in each orthant.
By definition, for $\mu=\tfrac12(a,b,c,d)\in\ZZ^4$ we have that
\begin{eqnarray}\label{eq5:q-odd-LL'1}
\mu\in\mathcal L &\Longleftrightarrow & a+(1+r)b+(1+2r)c+(1+4r)d\equiv 0\pmod{r^2}, \\
\mu\in\mathcal L'&\Longleftrightarrow & a+(1-r)b+(1-2r)c+(1-4r)d\equiv 0\pmod{r^2}, \notag
\end{eqnarray}
or equivalently,
\begin{eqnarray} \label{eq5:q-odd-LL'2}
\mu\in\mathcal L &\Longleftrightarrow & (a+b+c+d)+r(b+2c+4d)\equiv 0\pmod{r^2},\\
\mu\in\mathcal L'&\Longleftrightarrow & (a+b+c+d)-r(b+2c+4d)\equiv 0\pmod{r^2}, \notag
\end{eqnarray}
thus $a+b+c+d\equiv 0\pmod r$ in both cases.

Let us first examine the orthant ${+}{+}{+}{+}$.
Any element of $\ZZ^4$ in this orthant with one-norm equal to $k$ can be written as
\begin{equation}\label{eq5:q-odd-mu++++}
\mu=\tfrac12(a,b,c,d)=\tfrac12\big(\, 2(k-2-x)+1, \,2(x-y)+1, \,2(y-z)+1, \,2z+1\big),
\end{equation}
with $0\leq z\leq y\leq x\leq k-2$.
From \eqref{eq5:q-odd-LL'2}, if $\mu$ is in $\mathcal L$ or in $\mathcal L'$, then $a+b+c+d=2k\equiv 0\pmod r$, thus $k$ is divisible by $r$ since $r$ is odd.
From now on we assume $k=\omega r$ with $\omega\geq0$.
Now, \eqref{eq5:q-odd-LL'2} implies that $\mu$ as in \eqref{eq5:q-odd-mu++++} is in $\mathcal L$ (resp.\ $\mathcal L'$) if and only if
\begin{align}
2x+2y+4z+7+2\omega\equiv 0\pmod r &\label{eq5:q-odd-L++++}\\
(\text{resp.\quad}
2x+2y+4z+7-2\omega \equiv 0\pmod r & ). \notag
\end{align}

Set $\xi=e^{2\pi i/r}$.
The number $N_{\mathcal L}^{{+}{+}{+}{+}}(k)$ of elements $\mu\in\mathcal L$ contained in ${+}{+}{+}{+}$ satisfying $\norma{\mu}=k$ is equal to
\begin{align}\label{eq5:q-odd-N++++1}
N_{\mathcal L}^{{+}{+}{+}{+}}(k)
&= \sum_{x=0}^{k-2} \sum_{y=0}^{x} \sum_{z=0}^{y} \frac{1}{r} \sum_{l=0}^{r-1} \xi^{l(7+2\omega+2x+2y+4z)} \\
&= \frac{1}{r}\sum_{x=0}^{k-2} \sum_{y=0}^{x} \sum_{z=0}^{y}  1
    + \frac{1}{r} \sum_{l=1}^{r-1} \xi^{l(7+2\omega)}\sum_{x=0}^{k-2} \xi^{2lx}\sum_{y=0}^{x} \xi^{2ly}\sum_{z=0}^{y} \xi^{4lz}. \notag
\end{align}
Since $\xi^{4l}\neq1$ for any $1\leq l\leq r-1$, we obtain that
\begin{align}\label{eq5:q-odd-N++++2}
N_{\mathcal L}^{{+}{+}{+}{+}}(k)
&= \frac{1}{r}\binom{k+1}{3}
    + \frac{1}{r} \sum_{l=1}^{r-1} \xi^{l(7+2\omega)}\sum_{x=0}^{k-2} \xi^{2lx}\sum_{y=0}^{x} \xi^{2ly} \,\frac{\xi^{4l(y+1)}-1}{\xi^{4l}-1}\\
&= \frac{1}{r}\binom{k+1}{3}
    + \frac{1}{r} \sum_{l=1}^{r-1} \frac{\xi^{l(7+2\omega)}}{\xi^{4l}-1}
        \sum_{x=0}^{k-2} \xi^{2lx}\sum_{y=0}^{x} \big(\xi^{4l} \xi^{6ly}-\xi^{2ly}\big). \notag
\end{align}
Since $k=r\omega$ is divisible by $r$, it follows that
$$
\sum_{x=0}^{k-2} \xi^{2lx}\sum_{y=0}^{x} \xi^{2ly}
= \sum_{x=0}^{k-2} \xi^{2lx} \frac{\xi^{2l(x+1)}-1}{\xi^{2l}-1}
= \frac{-\xi^{2l}\xi^{-4l}+\xi^{-2l}}{\xi^{2l}-1}=0
$$
for every $1\leq l\leq r-1$.
Therefore
\begin{align}\label{eq5:q-odd-N++++3}
N_{\mathcal L}^{{+}{+}{+}{+}}(k)
&= \frac{1}{r}\binom{k+1}{3}
    + \frac{1}{r} \sum_{l=1}^{r-1} \frac{\xi^{l(7+2\omega)}\xi^{4l}}{\xi^{4l}-1}
        \sum_{x=0}^{k-2} \xi^{2lx}\sum_{y=0}^{x}  \xi^{6ly}.
\end{align}
From \eqref{eq5:q-odd-L++++}, the same formula holds for $N_{\mathcal L'}^{{+}{+}{+}{+}}(k)$ by replacing $\omega$ by $-\omega$.

Clearly
$$
\sum_{y=0}^{x}  \xi^{6ly}
= \begin{cases}
  \dfrac{\xi^{6l(x+1)}-1}{\xi^{6l}-1} &\quad\text{if } 3l\not\equiv0\pmod r,\\
  x+1 &\quad\text{if } 3l\equiv0\pmod r.
\end{cases}
$$
Then, one can check that $\sum_{x=0}^{k-2} \xi^{2lx}\sum_{y=0}^{x}  \xi^{6ly}= 0$ for any $l$ with $3l\not\equiv0\pmod r$.
Hence
\begin{equation}\label{eq5:q-odd-N++++4}
  N_{\mathcal L}^{{+}{+}{+}{+}}(k)=N_{\mathcal L'}^{{+}{+}{+}{+}}(k) = \frac{1}{r}\binom{k+1}{3}= \frac{1}{r}\binom{r\omega+1}{3}
  \qquad\text{if }r\not\equiv0\pmod3.
\end{equation}

We now assume $r\equiv0\pmod 3$.
Set $\nu =\xi^{r/3}$, which is a primitive root of unity of order $3$.
From the second to last equation and \eqref{eq5:q-odd-N++++3} we see that
\begin{align}\label{eq5:q-odd-N++++5}
N_{\mathcal L}^{{+}{+}{+}{+}}(k)
&= \frac{1}{r}\binom{k+1}{3}
    + \frac{1}{r} \sum_{l\in\{\frac{r}{3},\frac{2r}{3}\}} \frac{\xi^{l(7+2\omega)}}{1-\xi^{-4l}}
        \sum_{x=0}^{k-2} \xi^{2lx} (x+1)\\
&= \frac{1}{r}\binom{k+1}{3}
    + \frac{1}{r} \sum_{j=1}^{2} \frac{\nu^{j(1-\omega)}}{1-\nu^{-j}}
        \left(\sum_{\alpha=0}^{\omega \frac r3-1}\sum_{x_0=0}^{2} \nu^{-jx_0} (3\alpha+1+x_0) - k\nu^{2j(k-1)}\right)\notag\\
&= \frac{1}{r}\binom{k+1}{3}
    + \frac{1}{r} \sum_{j=1}^{2} \frac{\nu^{-j(1+\omega)}}{\nu^{j}-1}
        \left(\frac{\omega r}3(\nu^{-j}+2\nu^{j}) - \omega r\nu^{j}\right)\notag \\
&= \frac{1}{r}\binom{k+1}{3}
    +\frac{\omega}3 \big(\nu^{-(1+\omega)}+\nu^{1+\omega}\big)
    - \omega \,\frac{\nu^{-\omega}-\nu^{1+\omega}}{\nu-1} \notag
\end{align}
Similarly,
\begin{align}\label{eq5:q-odd-N++++5'}
N_{\mathcal L'}^{{+}{+}{+}{+}}(k)
&= \frac{1}{r}\binom{k+1}{3}
    +\frac{\omega}3 \big(\nu^{-(1-\omega)}+\nu^{1-\omega}\big)
    - \omega \,\frac{\nu^{\omega}-\nu^{1-\omega}}{\nu-1}.
\end{align}
By a straightforward computation we obtain that
\begin{equation}\label{eq5:q-odd-N++++6}
N_{\mathcal L}^{{+}{+}{+}{+}}(r\omega)=N_{\mathcal L'}^{{+}{+}{+}{+}}(r\omega)=
\begin{cases}
\displaystyle  \frac{1}{r}\binom{r\omega+1}{3}+\frac{2\omega}{3} \quad&\text{if }\omega\equiv0\pmod3,\\[4mm]
\displaystyle  \frac{1}{r}\binom{r\omega+1}{3}-\frac{\omega}{3}\quad&\text{if }\omega\not\equiv0\pmod3.
\end{cases}
\end{equation}

We next proceed with the orthants with an odd number of negative signs, namely ${+}{+}{+}{-}$, ${+}{+}{-}{+}$, ${+}{-}{+}{+}$ and ${-}{+}{+}{+}$.
Note that $\mu\in\mathcal L$ if and only if $-\mu\in\mathcal L$, thus it is not necessary to check the opposite orthants.
We will often use the notation $\lfloor x\rfloor=\max\{a\in\Z: a\leq x\}$ and $\lceil x\rceil=\min\{a\in\Z: a\geq x\}$ for $x\in\R$.

\begin{lemma}\label{lem5:q-oddRodd}
Let $r$ be an odd positive integer and let $k=\omega r+k_0$ such that $0\leq k_0<r$ and $\omega\geq0$.
Then, the number of element $\mu$ in $\mathcal L$ in the orthant ${+}{+}{+}{-}$, ${+}{+}{-}{+}$, ${+}{-}{+}{+}$ or ${-}{+}{+}{+}$ with $\norma{\mu}=k+2$ is equal to
\begin{equation}\label{eq5:q-oddRoddN}
\sum_{\gamma=0}^{\omega + \lfloor\tfrac{k_0-z_0}{r}\rfloor}
\sum_{y_0=0}^{r-1}
\sum_{\beta=\gamma+\lceil\tfrac{z_0-y_0}{r}\rceil}^{\omega + \lfloor\tfrac{k_0-y_0}{r}\rfloor}
\left(
\omega-\beta +1+ \lfloor\tfrac{k_0-x_{y_0,\gamma}}{r}\rfloor +\lfloor\tfrac{x_{y_0,\gamma}-y_0}{r}\rfloor
\right),
\end{equation}
where $z_0$ is the only integer that satisfies $0\leq z_0<r$ and $2z_0\equiv k+1\pmod r$ and $x_{y_0,\gamma}$ is the only integer such that $0\leq x_{y_0,\gamma}<r$ and
\begin{equation}\label{eq5:q-oddRoddx0}
\begin{cases}
2x_{y_0,\gamma}\equiv 1-2y_0+12z_0-2\big(\omega-2\gamma+\tfrac{k_0+1-2z_0}{r}\big)\pmod{r}
    &\quad\text{for \; ${+}{+}{+}{-}$}, \\
2x_{y_0,\gamma}\equiv -3-6y_0+12z_0-2\big(\omega-2\gamma+\tfrac{k_0+1-2z_0}{r}\big)\pmod{r}
    &\quad\text{for \; ${+}{+}{-}{+}$}, \\
4x_{y_0,\gamma}\equiv -5-4y_0+10z_0-2\big(\omega-2\gamma+\tfrac{k_0+1-2z_0}{r}\big)\pmod{r}
    &\quad\text{for \; ${+}{-}{+}{+}$}, \\
2x_{y_0,\gamma}\equiv -7-4y_0+8z_0-2k_0-2\big(\omega-2\gamma+\tfrac{k_0+1-2z_0}{r}\big)\pmod{r}
    &\quad\text{for \; ${-}{+}{+}{+}$}.
\end{cases}
\end{equation}
Formula \eqref{eq5:q-oddRoddN} also holds for $\mathcal L'$ replacing $x_{y_0,\gamma}$ by $x_{y_0,\gamma}'$, where $x_{y_0,\gamma}'$ is the only integer such that $0\leq x_{y_0,\gamma}'<r$ and
\begin{equation}\label{eq5:q-oddRoddx0'}
\begin{cases}
2x_{y_0,\gamma}\equiv 1-2y_0+12z_0+2\big(\omega-2\gamma+\tfrac{k_0+1-2z_0}{r}\big)\pmod{r}
    &\quad\text{for \; ${+}{+}{+}{-}$}, \\
2x_{y_0,\gamma}\equiv -3-6y_0+12z_0+2\big(\omega-2\gamma+\tfrac{k_0+1-2z_0}{r}\big)\pmod{r}
    &\quad\text{for \; ${+}{+}{-}{+}$}, \\
4x_{y_0,\gamma}\equiv -5-4y_0+10z_0+2\big(\omega-2\gamma+\tfrac{k_0+1-2z_0}{r}\big)\pmod{r}
    &\quad\text{for \; ${+}{-}{+}{+}$}, \\
2x_{y_0,\gamma}\equiv -7-4y_0+8z_0-2k_0+2\big(\omega-2\gamma+\tfrac{k_0+1-2z_0}{r}\big)\pmod{r}
    &\quad\text{for \; ${-}{+}{+}{+}$}.
\end{cases}
\end{equation}
\end{lemma}

\begin{proof}
Any element in $\ZZ^4$ in the orthant ${+}{+}{+}{-}$ with one-norm equal to $k+2$ can be written as
$$
\mu=\tfrac12(a,b,c,d)=\tfrac12\big(2(k-x)+1, 2(x-y)+1, 2(y-z)+1, -2z-1\big),
$$
with $0\leq z\leq y\leq x\leq k$ integer numbers.
From \eqref{eq5:q-odd-LL'2}, if $\mu$ is in $\mathcal L$ or in $\mathcal L'$, then $a+b+c+d=2k+2-4z\equiv 0\pmod r$, thus $2z\equiv k_0+1\pmod r$ since $r$ is odd.
We write $z=\gamma r+z_0$ with $0\leq z_0<r$.
The condition $0\leq z\leq k$ is equivalent to
\begin{equation}\label{eq5:intervalo-alpha}
\begin{array}{rcccl}
0&\leq& \gamma r+z_0 &\leq &\omega r+k_0, \\[1mm]
-\tfrac{z_0}{r}&\leq& \gamma &\leq& \omega+ \tfrac{k_0-z_0}{r}\\[1mm]
0=\lceil-\tfrac{z_0}{r}\rceil &\leq& \gamma &\leq& \omega+ \lfloor\tfrac{k_0-z_0}{r}\rfloor.
\end{array}
\end{equation}
Note that in \eqref{eq5:q-oddRoddN}, $\gamma$ runs in the same interval.

From \eqref{eq5:q-odd-LL'2}, we have that $\mu\in\mathcal L$ if and only if
\begin{align}
b+2c+4d &\equiv -\tfrac{2k+2-4z}{r} \pmod r \notag\\
2x+2y-12z_0-1 &\equiv -2\big(\omega-2\gamma+\tfrac{k_0+1-2z_0}{r} \big)\pmod r. \label{eq5:q-odd-L+++-}
\end{align}
Similarly, $\mu\in\mathcal L'$ if and only if
\begin{align}
2x+2y-12z_0-1 &\equiv 2\big(\omega-2\gamma+\tfrac{k_0+1-2z_0}{r} \big)\pmod r. \label{eq5:q-odd-L'+++-}
\end{align}
Let $y=\beta r+y_0$ with $0\leq y_0<r$ such that $z\leq y\leq k$, that is,
$\gamma+\lceil\tfrac{z_0-y_0}{r}\rceil \leq \beta \leq \omega + \lfloor\tfrac{k_0-y_0}{r}\rfloor$, which is the same interval as in \eqref{eq5:q-oddRoddN}.
From \eqref{eq5:q-odd-L+++-} and \eqref{eq5:q-odd-L'+++-}, we have that $y_0$ determines $x$ modulo $r$ if $\mu$ is in $\mathcal L$ or in $\mathcal L'$.

Suppose $\gamma$ and $y_0$ as above are fixed.
Let $x_{y_0,\gamma}$ (resp.\ $x_{y_0,\gamma}'$) be the only integer such that $0\leq x_{y_0,\gamma}<r$ (resp.\ $0\leq x_{y_0,\gamma}'<r$) and satisfies \eqref{eq5:q-odd-L+++-} (resp.\ \eqref{eq5:q-odd-L'+++-}).
Hence, any other solution of \eqref{eq5:q-odd-L+++-} (resp.\ \eqref{eq5:q-odd-L'+++-}) is $x=\alpha r+x_{y_0,\gamma}$ (resp.\ $x'=\alpha' r+x_{y_0,\gamma}'$).
One can check that the condition $y\leq x\leq k$ (resp.\ $y\leq x'\leq k$) is equivalent to
\begin{align*}
\beta + \lceil\tfrac{y_0-x_{y_0,\gamma}}{r}\rceil &\leq \alpha \leq \omega + \lfloor\tfrac{k_0-x_{y_0,\gamma}}{r}\rfloor &
\Big(\text{resp.\quad }
\beta + \lceil\tfrac{y_0-x_{y_0,\gamma}'}{r}\rceil &\leq \alpha' \leq \omega + \lfloor\tfrac{k_0-x_{y_0,\gamma}'}{r}\rfloor \Big),
\end{align*}
thus we have $\omega +1-\beta+ \lfloor\tfrac{k_0-x_{y_0,\gamma}}{r}\rfloor- \lceil\tfrac{y_0-x_{y_0,\gamma}}{r}\rceil$ (resp.\ $\omega +1-\beta+ \lfloor\tfrac{k_0-x_{y_0,\gamma}'}{r}\rfloor- \lceil\tfrac{y_0-x_{y_0,\gamma}'}{r}\rceil$) choices.
By adding over $\gamma$, $y_0$, $\beta$ the number of choices for $\alpha$ and $\alpha'$ respectively, we obtain the asserted formulas.

The other cases follow in the same manner.
One has to write the elements in the corresponding orthants as follows:
$$
\begin{array}{@{\tfrac12 \big(\;}c@{\;\;,\;\;}c@{\;\;,\;\;}c@{\;\;,\;\;}c@{\;\big)\qquad}l}
2(k-x)+1&   2(x-y)+1&   -2z-1&      2(y-z)+1&   \text{for ${+}{+}{-}{+}$},\\[1mm]
2(k-x)+1&   -2z-1&      2(x-y)+1&   2(y-z)+1&   \text{for ${+}{-}{+}{+}$},\\[1mm]
-2z-1&      2(k-x)+1&   2(x-y)+1&   2(y-z)+1&   \text{for ${-}{+}{+}{+}$},
\end{array}
$$
for $0\leq z\leq y\leq x\leq k$.
\end{proof}

We next prove that the formulas for $\mathcal L$ and $\mathcal L'$ in the previous lemma coincide in each orthant.
We restrict our attention to the case ${+}{+}{+}{-}$ since the other cases are very similar to this one.
From Lemma~\ref{lem5:q-oddRodd}, the difference of \eqref{eq5:q-oddRoddN} for $\mathcal L$ and $\mathcal L'$ is equal to
\begin{align}\label{eq5:q-oddD(k)+++-1}
D(k)&= \sum_{\gamma=0}^{\omega + \lfloor\tfrac{k_0-z_0}{r}\rfloor}
\sum_{y_0=0}^{r-1}
\sum_{\beta=\gamma+\lceil\tfrac{z_0-y_0}{r}\rceil}^{\omega + \lfloor\tfrac{k_0-y_0}{r}\rfloor}
\left(
\lfloor\tfrac{k_0-x_{y_0,\gamma}}{r}\rfloor +\lfloor\tfrac{x_{y_0,\gamma}-y_0}{r}\rfloor
-
\lfloor\tfrac{k_0-x_{y_0,\gamma}'}{r}\rfloor -\lfloor\tfrac{x_{y_0,\gamma}'-y_0}{r}\rfloor
\right) \\
    &= \sum_{\gamma=0}^{\omega + \lfloor\tfrac{k_0-z_0}{r}\rfloor}
\sum_{y_0=0}^{r-1}
\left(\omega -\gamma+1+ \lfloor\tfrac{k_0-y_0}{r}\rfloor-\lceil\tfrac{z_0-y_0}{r}\rceil \right)
\left(
\lfloor\tfrac{k_0-x_{y_0,\gamma}}{r}\rfloor +\lfloor\tfrac{x_{y_0,\gamma}-y_0}{r}\rfloor
\right)\notag\\
    &\quad
    - \sum_{\gamma=0}^{\omega + \lfloor\tfrac{k_0-z_0}{r}\rfloor}
\sum_{y_0=0}^{r-1}
\left(\omega -\gamma+1+ \lfloor\tfrac{k_0-y_0}{r}\rfloor-\lceil\tfrac{z_0-y_0}{r}\rceil \right)
\left(
\lfloor\tfrac{k_0-x_{y_0,\gamma}'}{r}\rfloor +\lfloor\tfrac{x_{y_0,\gamma}'-y_0}{r}\rfloor
\right).\notag
\end{align}
We now divide the problem into three cases, namely, $k_0$ odd, $k_0\neq r-1$ even, and $k_0=r-1$.

We first assume that $k_0$ is odd, thus $z_0=\tfrac{k_0+1}{2}$.
From \eqref{eq5:q-oddRoddx0} and \eqref{eq5:q-oddRoddx0'} we have that $x_{y_0,\gamma}$ and $x_{y_0,\gamma}'$ are the smallest non-negative integer numbers such that
\begin{align}\label{eq5:q-oddRoddx0+++-}
  2x_{y_0,\gamma} &\equiv 1-2y_0+12z_0-2\big(\omega-2\gamma\big)\pmod{r},\\
  2x_{y_0,\gamma}'&\equiv 1-2y_0+12z_0+2\big(\omega-2\gamma\big)\pmod{r}. \notag
\end{align}
This implies that $x_{y_0,\omega-\gamma}'=x_{y_0,\gamma}$ for every $0\leq \gamma\leq \omega$ and every $y_0$.
By \eqref{eq5:q-oddD(k)+++-1}, we have that
\begin{align*}
D(k)
&=
   \sum_{\gamma=0}^{\omega}
\sum_{y_0=0}^{r-1}
\left(\omega -\gamma+1+ \lfloor\tfrac{k_0-y_0}{r}\rfloor-\lceil\tfrac{z_0-y_0}{r}\rceil \right)
\left(
\lfloor\tfrac{k_0-x_{y_0,\gamma}}{r}\rfloor +\lfloor\tfrac{x_{y_0,\gamma}-y_0}{r}\rfloor
\right)\\
&\quad
    - \sum_{\gamma=0}^{\omega}
\sum_{y_0=0}^{r-1}
\left(\gamma+1+ \lfloor\tfrac{k_0-y_0}{r}\rfloor-\lceil\tfrac{z_0-y_0}{r}\rceil \right)
\left(
\lfloor\tfrac{k_0-x_{y_0,\gamma}}{r}\rfloor +\lfloor\tfrac{x_{y_0,\gamma}-y_0}{r}\rfloor
\right)\\
&=
\sum_{\gamma=0}^{\omega}
\left(\omega -2\gamma\right)
\sum_{y_0=0}^{r-1}
\left(
\lfloor\tfrac{k_0-x_{y_0,\gamma}}{r}\rfloor +\lfloor\tfrac{x_{y_0,\gamma}-y_0}{r}\rfloor
\right).
\end{align*}
Now, since $y_0\mapsto x_{y_0,\gamma}$ is a bijection of $\{x\in\Z: 0\leq x<r\}$ by \eqref{eq5:q-oddRoddx0+++-}, we have that $\sum_{y_0=0}^{r-1} \lfloor\tfrac{k_0-x_{y_0,\gamma}}{r}\rfloor=\sum_{y_0=0}^{r-1} \lfloor\tfrac{k_0-y_0}{r}\rfloor  = k_0+1-r $.
Moreover, $y_0\mapsto x_{y_0,\gamma}-y_0$ is also a bijection of $\{x\in\Z: 0\leq x<r\}$ by \eqref{eq5:q-oddRoddx0+++-}, so $\sum_{y_0=0}^{r-1} \lfloor\tfrac{x_{y_0,\gamma}-y_0}{r}\rfloor=\tfrac{1-r}{2}$.
Hence
$$
D(k)= \sum_{\gamma=0}^{\omega}
\left(\omega -2\gamma\right)
(k_0+1-r+\tfrac{1-r}{2}) =0.
$$

We next assume $k_0$ is even and different from $r-1$, thus $z_0=\tfrac{k_0+1+r}{2}\geq1$.
Similar to the previous case, by \eqref{eq5:q-oddRoddx0} and \eqref{eq5:q-oddRoddx0'} we have that $x_{y_0,\omega-1-\gamma}'=x_{y_0,\gamma}$ for every $0\leq \gamma\leq \omega-1$ and every $y_0$, and by \eqref{eq5:q-oddD(k)+++-1} one obtains that
\begin{align*}
  D(k)
&=
\sum_{\gamma=0}^{\omega-1}
\left(\omega-1 -2\gamma\right)
\sum_{y_0=0}^{r-1}
\left(
\lfloor\tfrac{k_0-x_{y_0,\gamma}}{r}\rfloor +\lfloor\tfrac{x_{y_0,\gamma}-y_0}{r}\rfloor
\right).
\end{align*}
Analysis similar to the case $k_0$ odd shows that the second sum does not depend on $\gamma$, thus $D(k)=0$.

Finally, we assume $k_0=r-1$, thus $z_0=0$.
In this case we obtain $x_{y_0,\omega+1-\gamma}'=x_{y_0,\gamma}$ for all $0\leq \gamma\leq \omega+1$, but the sum over $\gamma$ in \eqref{eq5:q-oddD(k)+++-1} runs over $[\![0,\omega ]\!]$.
By a similar argument as the one used in the two previous cases, all the terms in \eqref{eq5:q-oddD(k)+++-1} cancel in the sum over $\gamma$ except for $\gamma=0$.
Hence
\begin{align*}
D(k)
&=
\sum_{y_0=0}^{r-1}
\left(\omega +1+ \lfloor\tfrac{k_0-y_0}{r}\rfloor-\lceil\tfrac{z_0-y_0}{r}\rceil \right)
\left(
\lfloor\tfrac{k_0-x_{y_0,0}}{r}\rfloor +\lfloor\tfrac{x_{y_0,0}-y_0}{r}\rfloor
-\lfloor\tfrac{k_0-x_{y_0,0}'}{r}\rfloor -\lfloor\tfrac{x_{y_0,0}'-y_0}{r}\rfloor
\right)\notag\\
&=
\left(\omega +1\right)
\sum_{y_0=0}^{r-1}
\left(
\lfloor\tfrac{x_{y_0,0}-y_0}{r}\rfloor - \lfloor\tfrac{x_{y_0,0}'-y_0}{r}\rfloor
\right)\notag.
\end{align*}
In the last step we computed many floors by using the fact that $k_0=r-1$ and $z_0=0$.
We already know that $\sum_{y_0=0}^{r-1} \lfloor\tfrac{x_{y_0,0}-y_0}{r}\rfloor =\sum_{y_0=0}^{r-1} \lfloor\tfrac{x_{y_0,0}'-y_0}{r}\rfloor =\tfrac{1-r}{2}$, thus $D(k)=0$.
This completes the proof in the orthant ${+}{+}{+}{-}$.
Very similar arguments work in the orthants ${+}{+}{-}{+}$, ${+}{-}{+}{+}$ and ${-}{+}{+}{+}$.

We now consider the orthants with an even number of negative entries.
We leave the case ${+}{+}{+}{+}$ to the end.

\begin{lemma}\label{lem5:q-oddReven}
Let $r$ be an odd positive integer and let $k=\omega r+k_0$ such that $0\leq k_0<r$ and $\omega\geq0$.
Then, the number of element $\mu$ in $\mathcal L$ in the orthant ${+}{+}{-}{-}$, ${+}{-}{+}{-}$ or ${+}{-}{-}{+}$ with $\norma{\mu}=k+2$ is equal to
\begin{equation}\label{eq5:q-oddRevenN}
\sum_{\beta=0}^{\omega + \lfloor\tfrac{k_0-y_0}{r}\rfloor}
\sum_{z_0=0}^{r-1}
\sum_{\gamma=0}^{\beta + \lfloor\tfrac{y_0-z_0}{r}\rfloor}
\left(
\omega-\beta +1+ \lfloor\tfrac{k_0-x_{z_0,\beta}}{r}\rfloor +\lfloor\tfrac{x_{z_0,\beta}-y_0}{r}\rfloor
\right),
\end{equation}
where $y_0$ is the only integer that satisfies $0\leq y_0<r$ and $2y_0\equiv k\pmod r$ and $x_{z_0,\beta}$ is the only integer such that $0\leq x_{z_0,\beta}<r$ and
\begin{equation}\label{eq5:q-oddRevenx0}
\begin{cases}
2x_{z_0,\beta}\equiv 5+6y_0+4z_0 -2\big(\omega-2\beta+\tfrac{k_0+1-2y_0}{r}\big)\pmod{r}
    &\quad\text{for \; ${+}{+}{-}{-}$}, \\
4x_{z_0,\beta}\equiv 3+6y_0+6z_0 -2\big(\omega-2\beta+\tfrac{k_0+1-2y_0}{r}\big)\pmod{r}
    &\quad\text{for \; ${+}{-}{+}{-}$}, \\
8x_{z_0,\beta}\equiv -1+10y_0+2z_0 -2\big(\omega-2\beta+\tfrac{k_0+1-2y_0}{r}\big)\pmod{r}
    &\quad\text{for \; ${+}{-}{-}{+}$}.
\end{cases}
\end{equation}
Formula \eqref{eq5:q-oddRevenN} also holds for $\mathcal L'$ when replacing $x_{z_0,\beta}$ by $x_{z_0,\beta}'$, where $x_{z_0,\beta}'$ is the only integer such that $0\leq x_{z_0,\beta}'<r$ and
\begin{equation}\label{eq5:q-oddRevenx0'}
\begin{cases}
2x_{z_0,\beta}\equiv 5+6y_0+4z_0 +2\big(\omega-2\beta+\tfrac{k_0+1-2y_0}{r}\big)\pmod{r}
    &\quad\text{for \; ${+}{+}{-}{-}$}, \\
4x_{z_0,\beta}\equiv 3+6y_0+6z_0 +2\big(\omega-2\beta+\tfrac{k_0+1-2y_0}{r}\big)\pmod{r}
    &\quad\text{for \; ${+}{-}{+}{-}$}, \\
8x_{z_0,\beta}\equiv -1+10y_0+2z_0 +2\big(\omega-2\beta+\tfrac{k_0+1-2y_0}{r}\big)\pmod{r}
    &\quad\text{for \; ${+}{-}{-}{+}$}.
\end{cases}
\end{equation}
\end{lemma}

\begin{proof}
The proof of this lemma is very similar to that one of Lemma~\ref{lem5:q-oddRodd}.
One has to write the elements with one-norm equal to $k+2$ in the orthants follows:
$$
\begin{array}{@{\tfrac12 \big(\;}c@{\;\;,\;\;}c@{\;\;,\;\;}c@{\;\;,\;\;}c@{\;\big)\qquad}l}
2(k-x)+1&   2(x-y)+1&   -2(y-z)-1&  -2z-1&      \text{for ${+}{+}{-}{-}$},\\[1mm]
2(k-x)+1&   -2(y-z)-1&  2(x-y)+1&   -2z-1&      \text{for ${+}{-}{+}{-}$},\\[1mm]
2(k-x)+1&   -2(y-z)-1&  -2z-1&      2(x-y)+1&   \text{for ${+}{-}{-}{+}$}
\end{array}
$$
for $0\leq z\leq y\leq x\leq k$.
\end{proof}

Moreover, the equality between the number of elements in $\mathcal L$ and in $\mathcal L'$ with one-norm $k+2$ in the orthants ${+}{+}{-}{-}$, ${+}{-}{+}{-}$ and ${+}{-}{-}{+}$ follows as before.
This concludes the proof.
\end{proof}

\begin{remark}
  The proof of Theorem~\ref{thm5:q-odd} is quite involved since it gives formulas for $N_{\mathcal L}(r,k)$ for every $k$ and $r$.
  This allows us to compute the Dirac spectrum explicitly by Theorem~\ref{thm4:dimV-lens}.
  The skeptical reader can, by using a computer, check  that the formulas are right for low values of $r$ and $k$.
\end{remark}

\begin{remark}
The computations included in the next section hint at Theorem~\ref{thm5:q-odd} being generalizable to the pairs $L=L(r^2t; 1,1+rt, 1+2rt, 1+4rt)$ and $L'=L(r^2t; 1,1-rt, 1-2rt, 1-4rt)$ where $r$ is again an odd positive integer and $t$ is an arbitrary positive integer.
When $t$ is even, the order of the fundamental group $q=r^2t$ is even, thus $L$ and $L'$ have two different spin structures, namely, $\tau_0$, $\tau_1$ and $\tau_0'$, $\tau_1'$ respectively.
Then, $L$ with $\tau_0$ and $L'$ with $\tau_0'$ are Dirac isospectral and so are $L$ with $\tau_1$ and $L'$ with $\tau_1'$.
\end{remark}

\section{Computational examples}\label{sec:comp}
Besides the infinite families in Theorems~\ref{thm5:increasing}, \ref{thm5:q-even} and \ref{thm5:q-odd}, there are many more examples of Dirac isospectrality in the class of lens spaces.
To make a characterization of Dirac isospectrality presented herein accessible by computer computations, we first present a finiteness result.

Let $q$ be a positive integer.
We let
$$
C(q):=\{\tfrac12(a_1,\dots,a_m)\in\ZZ^m: |\tfrac{a_j}{2}|<q \quad\forall \, j\}.
$$
We call \emph{$q$-reduced} to the elements in $C(q)$ and for any subset $\mathcal L$ of $\ZZ^m$ we let
\begin{equation*}
  N_{\mathcal L}^{\mathrm{red}} (\epsilon,k) = \#\{\mu\in\mathcal L\cap C(q): \norma{\mu}=k+\tfrac m2,\; R(\mu)\equiv\epsilon\pmod2\}.
\end{equation*}
Note that $N_{\mathcal L}^{\mathrm{red}} (\epsilon,k) =0$ for every $k\geq mq-\tfrac m2$, so there are only finitely many non-zero of these numbers.

\begin{proposition}\label{prop6:finiteness}
Let $\mathcal L$ be an affine congruence lattice of a spin lens space of dimension $2m-1$ with fundamental group of order $q$.
Then
\begin{equation}\label{eq6:N^red_formula}
N_{\mathcal L} (\epsilon,k)
= \sum_{\beta=0}^{\lfloor{k}/{q}\rfloor} \tbinom{\beta+m-1}{m-1}\; N_{\mathcal L}^{\mathrm{red}}(\epsilon,k-\beta q).
\end{equation}
\end{proposition}
\begin{proof}
We recall that $N_{\mathcal L} (\epsilon,k)=\#\{\mu\in\mathcal L: \norma{\mu}=k+\frac m2,\; R(\mu)\equiv\epsilon\pmod2\}$.
Write $k=\alpha q+k_0$ with $0\leq k_0<q$, thus $\lfloor{k}/{q}\rfloor=\alpha$.
If $\alpha=0$, \eqref{eq6:N^red_formula} reduces to $N_{\mathcal L} (\epsilon,k_0) = N_{\mathcal L}^{\mathrm{red}}(\epsilon,k_0)$, which is true since every $\mu\in\ZZ^m$ with $\norma{\mu}=k_0+\tfrac m2<q+\tfrac m2$ is $q$-reduced.
We now assume $\alpha>0$.
Note that $\mu+q\Z^m\subset \mathcal L$ for every $\mu\in\mathcal L$ by \eqref{eq3:L(q;s)} and \eqref{eq3:L(q;s;h)}.
On the one hand, for every $\mu\in\ZZ^m$ there is exactly one element $\mu_0\in C(q)$ in the same orthant such that $\mu\in\mu_0+q\Z^m$.
Clearly, $\norma{\mu_0}=(\alpha-\beta)r+k_0+\tfrac m2$ for some $0\leq \beta\leq\alpha$ if $\norma{\mu}=\alpha q+k_0+\tfrac m2$.

On the other hand, each $q$-reduced element $\mu_0=\tfrac12(a_1,\dots,a_m)\in\mathcal L$ with $\norma{\mu_0}=(\alpha-\beta)q+k_0+\tfrac m2$ is related to exactly $\binom{\beta+m-1}{m-1}$ elements $\mu$ with $\norma{\mu}=\alpha q+k_0+\tfrac m2$ in the same orthant of $\ZZ^m$ as $\mu_0$.
Indeed, $\mu=\mu_0+\eta$ for each
$$
\eta=(\tfrac{a_1}{|a_1|} \,h_1 q,\dots, \tfrac{a_m}{|a_m|} \, h_mq)
$$
with $h_j\geq0$ and $\beta=h_1+\dots+h_m$.
The number of choices for $\eta$ is the number of ways to write $\beta$ as the ordered sum of $m$ non-negative integers (ordered partitions), namely $\binom{\beta+m-1}{m-1}$.
This establishes \eqref{eq6:N^red_formula}.
\end{proof}

\begin{corollary}\label{cor6:finiteness-eigenv}
Two $2m-1$-dimensional spin lens spaces $(L,\tau)$ and $(L',\tau')$ with fundamental group of order $q$ are Dirac isospectral if and only if the multiplicities $\mult_{(L,\tau)}(\pm\lambda_k)$ and $\mult_{(L',\tau')}(\pm\lambda_k)$ of the Dirac eigenvalues $\pm\lambda_k=\pm (k+\tfrac{2m-1}{2})$ for $(L,\tau)$ and $(L',\tau')$ respectively coincide for all $0\leq k<mq$.
\end{corollary}
\begin{proof}
By Proposition~\ref{prop6:finiteness}, the numbers $N_{\mathcal L}^{\mathrm{red}}(\epsilon,k)$ for every $k<qm$ determine $N_{\mathcal L}(\epsilon,k)$ for every $k$.
One can check that the converse also holds.
Then, the assertion follows by Corollary~\ref{cor4:char-lens}.
\end{proof}

Now Corollary~\ref{cor4:char-lens} together with Proposition~\ref{prop6:finiteness} allows us to find all Dirac isospectral lens spaces for small values of $m$ and $q$ with the help of a computer.
The algorithm that we used to find Dirac isospectral lens spaces is straightforward. For given $m\ge 2$, $q\in\N$, generate a complete list of representations of isometry classes of lens spaces with a chosen orientation and spin structure of dimension $2m-1$ whose fundamental group is of order $q$ by using Proposition~\ref{prop2:lens-isom}. Next, calculate for all $0\le k< mq$ and $\epsilon\in\{0,1\}$ the numbers $N_{\mathcal L}^{\textrm{red}}(\epsilon,k)$ for all previously identified lens spaces. The last step is to partition the set of lens spaces into isospectral families by comparing all the numbers $N_{\mathcal{L}}^{\textrm{red}}(\epsilon,k)$ for every pair of lens spaces.

We have found examples of Dirac isospectral lens spaces using the above procedure.
In Table~\ref{table:examples} we show all of the examples for $n=7,q\leq 100$, for $n=11,q\leq 50$, for $n=15,q\leq 60$ and for $n=19,q\leq 40$.

\begin{table}\label{table:examples}
\caption{Examples of Dirac isospectral spin lens spaces in dimensions $n=7,11,15,19$ for low values of $q$.}
$
\begin{array}{r@{}r@{;\;}r@{,}r@{,\,}r@{,}r@{}r@{\qquad} r@{}r@{;\;}r@{,}r@{,\,}r@{,}r@{}r@{\qquad} r@{}r@{;\;}r@{,}r@{,\,}r@{,}r@{}r}
\multicolumn{21}{c}{\text{Dimension }n=7\rule{0pt}{0pt}}\\ \hline
L(& 32 & 1 &  3 &  5 & 15)&_{\tau_0} & L(& 49 & 1 &  6 &  8 & 22)&          & L(& 64 & 1 &  7 &  9 & 31)&_{\tau_0} \\
L(& 32 & 1 &  3 &  5 & 15)&_{\tau_1} & L(& 49 & 1 &  6 &  8 & 20)&          & L(& 64 & 1 &  7 &  9 & 31)&_{\tau_1} \\ \hline
L(& 75 & 1 &  4 & 14 & 16)&          & L(& 75 & 1 &  4 & 11 & 34)&          & L(& 80 & 1 &  3 &  9 & 27)&_{\tau_0} \\
L(& 75 & 1 &  4 & 11 & 19)&          & L(& 75 & 1 &  4 & 14 & 31)&          & L(& 80 & 1 &  9 & 13 & 37)&_{\tau_0} \\ \hline
L(& 81 & 1 &  8 & 19 & 37)&          & L(& 81 & 1 &  8 & 10 & 28)&          & L(& 81 & 1 &  8 & 10 & 37)&          \\
L(& 81 & 1 &  8 & 26 & 37)&          & L(& 81 & 1 &  8 & 10 & 26)&          & L(& 81 & 1 &  8 & 10 & 35)&          \\ \hline
L(& 96 & 1 & 11 & 13 & 47)&_{\tau_0} & L(& 98 & 1 & 13 & 15 & 43)&_{\tau_0} & L(& 98 & 1 & 13 & 15 & 41)&_{\tau_0} \\
L(& 96 & 1 & 11 & 13 & 47)&_{\tau_1} & L(& 98 & 1 & 13 & 15 & 41)&_{\tau_1} & L(& 98 & 1 & 13 & 15 & 43)&_{\tau_1} \\ \hline
\end{array}
$

$
\begin{array}{r@{}r@{;\;}r@{,}r@{,\,}r@{,}r@{,}r@{,}r@{}r@{\qquad} r@{}r@{;\;}r@{,}r@{,\,}r@{,}r@{,}r@{,}r@{}r@{\qquad} r@{}r@{;\;}r@{,}r@{,}r@{,}r@{,\,}r@{,}r@{}r}
\multicolumn{27}{c}{\text{Dimension }n=11\rule{0pt}{16pt}}\\ \hline
L(&40 & 1 & 1 &  1 & 11 & 11 & 11)&_{\tau_0} & L(&40 & 1 & 1 & 11 & 11 & 13 & 17)&_{\tau_0} & L(&44 & 1 & 3 &  5 &  7 &  9 & 19)&_{\tau_0} \\
L(&40 & 1 & 1 &  9 & 11 & 11 & 19)&_{\tau_0} & L(&40 & 1 & 1 &  3 &  7 & 11 & 11)&_{\tau_0} & L(&44 & 1 & 3 &  5 &  7 & 13 & 15)&_{\tau_0} \\
\multicolumn{9}{c}{}                         & L(&40 & 1 & 3 &  7 &  9 & 11 & 19)&_{\tau_0} &                                              \\ \hline
L(&44 & 1 & 3 &  5 &  7 &  9 & 19)&_{\tau_1} & L(&48 & 1 & 1 &  5 &  7 &  7 & 13)&_{\tau_0} & L(&48 & 1 & 1 &  7 &  7 & 17 & 23)&_{\tau_0} \\
L(&44 & 1 & 3 &  5 &  7 & 13 & 15)&_{\tau_1} & L(&48 & 1 & 5 &  7 & 11 & 13 & 19)&_{\tau_0} & L(&48 & 1 & 1 &  1 &  7 &  7 &  7)&_{\tau_0} \\
\multicolumn{9}{c}{}                         & L(&48 & 1 & 1 &  7 &  7 & 11 & 19)&_{\tau_0} & \multicolumn{9}{c}{}                         \\ \hline
\end{array}
$

$
\begin{array}{r@{}r@{;\;}r@{,}r@{,\,}r@{,}r@{,}r@{,}r@{,}r@{,}r@{}r@{\qquad} r@{}r@{;\;}r@{,}r@{,\,}r@{,}r@{,}r@{,}r@{,}r@{,}r@{}r}
\multicolumn{22}{c}{\text{Dimension }n=15\rule{0pt}{16pt}}\\ \hline
L(&39& 1 & 2 & 4 & 5 &  7 & 10 & 14 & 16)&          & L(&52& 1 & 3 & 5 & 7 &  9 & 11 & 17 & 25)&_{\tau_0} \\
L(&39& 1 & 2 & 4 & 7 &  8 & 10 & 16 & 17)&          & L(&52& 1 & 3 & 5 & 7 &  9 & 15 & 23 & 25)&_{\tau_1} \\ \hline
L(&52& 1 & 3 & 5 & 7 &  9 & 11 & 19 & 21)&_{\tau_1} & L(&52& 1 & 3 & 5 & 7 &  9 & 11 & 19 & 21)&_{\tau_0} \\
L(&52& 1 & 3 & 5 & 7 &  9 & 11 & 17 & 23)&_{\tau_1} & L(&52& 1 & 3 & 5 & 7 &  9 & 11 & 17 & 23)&_{\tau_0} \\ \hline
L(&52& 1 & 3 & 5 & 7 &  9 & 15 & 23 & 25)&_{\tau_0} & L(&56& 1 & 3 & 5 & 9 & 11 & 13 & 19 & 23)&_{\tau_0} \\
L(&52& 1 & 3 & 5 & 7 &  9 & 11 & 17 & 25)&_{\tau_1} & L(&56& 1 & 3 & 5 & 9 & 11 & 13 & 15 & 27)&_{\tau_0} \\ \hline
L(&56& 1 & 3 & 5 & 9 & 11 & 13 & 15 & 27)&_{\tau_1} \\
L(&56& 1 & 3 & 5 & 9 & 11 & 13 & 19 & 23)&_{\tau_1} \\ \hline
\end{array}
$

$
\begin{array}{r@{}r@{;\;}r@{,}r@{,\,}r@{,}r@{,}r@{,}r@{,}r@{,}r@{,}r@{,}r@{}r@{\qquad} r@{}r@{;\;}r@{,}r@{,\,}r@{,}r@{,}r@{,}r@{,}r@{,}r@{,}r@{,}r@{}r}
\multicolumn{24}{c}{\text{Dimension }n=19\rule{0pt}{16pt}}\\ \hline
L(&24 & 1 & 1 & 1 &  1 &   1 &  5 &  5 &  5 &  5 &  5)&_{\tau_0} & L(&40 & 1 & 1 & 1 &  9 &  9 & 11 & 11 & 11 & 19 & 19)&_{\tau_0} \\
L(&24 & 1 & 1 & 1 &  5 &   5 &  5 &  7 &  7 & 11 & 11)&_{\tau_0} & L(&40 & 1 & 1 & 1 &  1 &  1 & 11 & 11 & 11 & 11 & 11)&_{\tau_0} \\
L(&24 & 1 & 1 & 1 &  1 &   5 &  5 &  5 &  5 &  7 & 11)&_{\tau_0} & L(&40 & 1 & 1 & 1 &  1 &  9 & 11 & 11 & 11 & 11 & 19)&_{\tau_0} \\ \hline
L(&40 & 1 & 1 & 1 &  3 &   7 &  9 & 11 & 11 & 11 & 19)&_{\tau_0} & L(&40 & 1 & 1 & 3 &  3 &  7 &  7 &  9 & 11 & 11 & 19)&_{\tau_0} \\
L(&40 & 1 & 1 & 1 &  1 &   3 &  7 & 11 & 11 & 11 & 11)&_{\tau_0} & L(&40 & 1 & 1 & 3 &  7 &  9 & 11 & 11 & 13 & 17 & 19)&_{\tau_0} \\
L(&40 & 1 & 1 & 3 &  7 &   9 &  9 & 11 & 11 & 19 & 19)&_{\tau_0} & L(&40 & 1 & 1 & 3 &  3 &  7 &  7 & 11 & 11 & 13 & 17)&_{\tau_0} \\
L(&40 & 1 & 1 & 1 &  9 &  11 & 11 & 11 & 13 & 17 & 19)&_{\tau_0} & L(&40 & 1 & 1 & 1 &  3 &  3 &  7 &  7 & 11 & 11 & 11)&_{\tau_0} \\
L(&40 & 1 & 1 & 1 &  1 &  11 & 11 & 11 & 11 & 13 & 17)&_{\tau_0} & L(&40 & 1 & 1 & 1 &  3 &  7 & 11 & 11 & 11 & 13 & 17)&_{\tau_0} \\
\multicolumn{13}{c}{}                                            & L(&40 & 1 & 1 & 1 & 11 & 11 & 11 & 13 & 13 & 17 & 17)&_{\tau_0} \\ \hline
\end{array}
$
\end{table}

Surprisingly, we have found no example of Dirac isospectral lens spaces in dimensions $n\equiv 1\pmod4$.
Note that the order $q$ of the fundamental group of a lens space which is spin and of dimension $n\equiv 1\pmod4$ is necessarily odd.
We have checked the non-existence for $q\leq 1001$ in dimension $n=5$, for $q\leq 501$ in dimension $n=9$, for $q\leq 251$ in dimension $n=13$, and for $q\leq 125$ in dimension $n=17$.
It is thus reasonable to conjecture the non-existence of such examples in general.

\begin{conjecture}\label{conj6:dim5}
Two Dirac isospectral spin lens spaces of dimension $n\equiv 1\pmod4$ are necessarily isometric.
\end{conjecture}

The non-existence of such examples contrasts with the large number of known examples in the case of flat manifolds.
Miatello and Podest\'a showed in \cite[\S4]{MP06} a rich number of examples in any dimension $n\geq 4$.
In particular, they show a family of pairwise non-homeomorphic even dimensional Dirac isospectral compact flat manifolds with cardinality depending exponentially on the dimension.

We end this paper with a couple of remarks on the computational examples in dimensions congruent to $3$ modulo $4$.

\begin{remark}\label{rem6:dim7}
There seems to be a relation between Dirac isospectrality and $p$-isospectrality for all $p$ among $7$-dimensional lens spaces.
Most of the examples of Dirac isospectral lens spaces in dimension $7$ are also $p$-isospectral for all $p$.
For example, the pairs of lens spaces in the family given in Theorem~\ref{thm5:q-odd} are $p$-isospectral for all $p$.
Indeed, one can check that $(0,1,2,4)$ is hereditarily good mod any odd $r$ in the sense of \cite[Def.~1]{DD}, thus by \cite[Thm.~1]{DD}, the lens spaces $L(r^2t;1,1+rt,1+2rt,1+4rt)$ and $L(r^2t;1,1-rt,1-2rt,1-4rt)$ are $p$-isospectral for all $p$, for every odd $r$ and every $t\geq1$.

However, we have found the following exceptions:
$$
\begin{array}{l@{\qquad}l@{\qquad}l}
L(75; 1, 4, 14, 16) & L(150; 1, 11, 29, 31) & L(300; 1, 19, 41, 79) \\
L(75; 1, 4, 11, 19) & L(150; 1, 11, 19, 41) & L(300; 1, 19, 59, 61) \\[2mm]
\rule{0pt}{18pt}
L(75; 1, 4, 11, 34) & L(150; 1, 11, 31, 59) & L(300; 1, 19, 41, 139) \\
L(75; 1, 4, 14, 31) & L(150; 1, 11, 19, 71) & L(300; 1, 19, 59, 121).
\end{array}
$$
These are pairs of Dirac isospectral lens spaces which are not $p$-isospectral for any $p$.
Here, when $q$ is even, the lens spaces in each pair are Dirac isospectral if they are both equipped with $\tau_0$ or $\tau_1$ respectively.

In the inverse direction there are exceptions as well.
The lens spaces $L(100;1, 9, 11, 29)$ and $L(100; 1, 9, 11, 31)$ are $p$-isospectral for all $p$, but not Dirac isospectral for any choice of combination of spin structures.
The same applies to the lens spaces $L(100; 1, 9, 21, 39)$ and $L(100; 1, 9, 29, 31)$.
\end{remark}

\begin{remark}
In \cite{LMR} there are the first examples of non-isometric compact Riemannian manifolds which are isospectral with respect to the Hodge-Laplace operator acting on $p$-forms for every $p$, but are not strongly isospectral (isospectral for every natural differential operator on a natural bundle).
These examples show that all $p$-spectra together do not determine the spectra of all natural operators.

The pairs of non-isometric Dirac isospectral lens spaces in Theorem~\ref{thm5:q-odd} are also $p$-isospectral for every $p$.
Hence, this example shows that the Laplace spectra over all fundamental vector bundles do not determine the spectra of all natural operators.
In other words and in the notation of \cite[\S2,\S8]{LMR}, these lens spaces are $\tau$-isospectral for every fundamental representation $\tau$ of $K\simeq\Spin(2m-1)$ but are not strongly isospectral.
The fundamental representations of $\Spin(2m-1)$ are $\tau_{1},\dots,\tau_{m-1}$.
If $1\leq p\leq m-2$, then $\tau_p$ has highest weight $\varepsilon_1+\dots+\varepsilon_{p}$ and induces the vector bundle of $p$-forms and $\tau_{m-1}$ has highest weight $\tfrac12(\varepsilon_1+\dots+\varepsilon_{m-1})$ and induces the spinor bundle.

\end{remark}

\begin{remark}
In \cite{DD} the examples of \cite{LMR} were extended and simplified.
In particular, the authors gave a sufficient condition on the parameters of two lens spaces in order for them to be $p$-isospectral for all $p$.
During the work on this paper the authors wondered whether something similar as in \cite{DD} can be done in the Dirac case.
\end{remark}

\begin{remark}
All the examples of Dirac isospectral lens spaces we have found are homotopy-equivalent. Note that this is very different from the Laplace-Beltrami case, where Ikeda~\cite{Ik80} found isospectral non-homotopy equivalent lens spaces in dimension seven with $q = 13$. However, the reasons for this phenomenon remain still obscure to the authors.
\end{remark}

\bibliographystyle{plain}

\end{document}